\documentclass[11pt]{amsart}

\usepackage{anysize} \marginsize{1.3in}{1.3in}{1in}{1in}
\usepackage{comment}
\usepackage{xcolor}
\usepackage{amsmath}
\usepackage{mathtools}
\usepackage[all]{xy}
\usepackage[utf8]{inputenc}
\usepackage{varioref}
\usepackage{amsfonts}
\usepackage{amssymb}
\usepackage{bbm}
\usepackage{esint}
\usepackage{graphicx}
\usepackage{tikz}
\usepackage{empheq}
\usepackage{enumitem}
\usepackage{tikz-cd}
\usetikzlibrary{matrix,arrows,decorations.pathmorphing}
\usepackage{mathrsfs}
\usepackage[hypertexnames=false,backref=page,pdftex,
 	pdfpagemode=UseNone,
 	breaklinks=true,
 	extension=pdf,
 	colorlinks=true,
 	linkcolor=blue,
 	citecolor=red,
 	urlcolor=blue,
 ]{hyperref}

\newcommand{\PH}{{\operatorname{PH}}}

\newcommand{\sC}{{\mathcal C}}
\newcommand{\sD}{{\mathcal D}}

\newcommand{\sF}{{\mathcal F}}

\newcommand{\sH}{{\mathcal H}}

\newcommand{\sO}{{\mathcal O}}

\newcommand{\sV}{{\mathcal V}}

\newcommand{\sX}{{\mathcal X}}
\newcommand{\sY}{{\mathcal Y}}
\newcommand{\sZ}{{\mathcal Z}}

\newcommand{\scrB}{{\mathscr B}}
\newcommand{\scrC}{{\mathscr C}}
\newcommand{\scrD}{{\mathscr D}}

\newcommand{\scrX}{{\mathscr X}}

\newcommand{\scrZ}{{\mathscr Z}}

\newcommand{\C}{{\mathbb C}}

\newcommand{\HH}{{\mathbb H}}

\newcommand{\N}{{\mathbb N}}
\renewcommand{\P}{{\mathbb P}}
\newcommand{\Q}{{\mathbb Q}}
\newcommand{\R}{{\mathbb R}}

\newcommand{\gothm}{{\mathfrak m}}

\newcommand{\Art}{\operatorname{Art}}
\newcommand{\Aut}{\operatorname{Aut}}

\newcommand{\codim}{\operatorname{codim}}

\newcommand{\Cx}{\mathcal C^{\infty}_X}

\newcommand{\PHx}{\mathrm{PH}_X}

\newcommand{\Def}{{\operatorname{Def}}}

\newcommand{\di}{\partial}

\newcommand{\Hilb}{{\rm Hilb}}

\newcommand{\Hom}{\operatorname{Hom}}
\newcommand{\sHom}{{\sH om}}

\newcommand{\id}{{\rm id}}
\newcommand{\img}{\operatorname{im}}

\newcommand{\isom}{{\ \cong\ }}

\newcommand{\lt}{{\rm{lt}}}
\newcommand{\ol}[1]{{\overline{#1}}}

\newcommand{\red}{{\operatorname{red}}}
\newcommand{\reg}{{\operatorname{reg}}}

\newcommand{\rk}{{\rm rk}}

\newcommand{\sing}{{\operatorname{sing}}}

\newcommand{\Spec}{\operatorname{Spec}}

\newcommand{\SU}{{\rm SU}}

\newcommand{\Sym}{{\rm Sym}}

\renewcommand{\to}[1][]{\xrightarrow{\ #1\ }}

\newcommand{\tensor}{\otimes}
\newcommand{\tr}{{\rm tr}}

\newcommand{\vphi}{\varphi}
\newcommand{\vrho}{{\varrho}}

\newcommand{\wh}[1]{{\widehat{#1}}}
\newcommand{\wt}[1]{{\widetilde{#1}}}

\newcommand{\om}{\omega}

\makeatletter
\newcommand*{\da@rightarrow}{\mathchar"0\hexnumber@\symAMSa 4B }
\newcommand*{\da@leftarrow}{\mathchar"0\hexnumber@\symAMSa 4C }
\newcommand*{\xdashrightarrow}[2][]{%
  \mathrel{%
    \mathpalette{\da@xarrow{#1}{#2}{}\da@rightarrow{\,}{}}{}%
  }%
}
\newcommand{\xdashleftarrow}[2][]{%
  \mathrel{%
    \mathpalette{\da@xarrow{#1}{#2}\da@leftarrow{}{}{\,}}{}%
  }%
}
\newcommand*{\da@xarrow}[7]{%
  \sbox0{$\ifx#7\scriptstyle\scriptscriptstyle\else\scriptstyle\fi#5#1#6\m@th$}%
  \sbox2{$\ifx#7\scriptstyle\scriptscriptstyle\else\scriptstyle\fi#5#2#6\m@th$}%
  \sbox4{$#7\dabar@\m@th$}%
  \dimen@=\wd0 %
  \ifdim\wd2 >\dimen@
    \dimen@=\wd2 %
  \fi
  \count@=2 %
  \def\da@bars{\dabar@\dabar@}%
  \@whiledim\count@\wd4<\dimen@\do{%
    \advance\count@\@ne
    \expandafter\def\expandafter\da@bars\expandafter{%
      \da@bars
      \dabar@ 
    }%
  }%
  \mathrel{#3}%
  \mathrel{%
    \mathop{\da@bars}\limits
    \ifx\\#1\\%
    \else
      _{\copy0}%
    \fi
    \ifx\\#2\\%
    \else
      ^{\copy2}%
    \fi
  }%
  \mathrel{#4}%
}
\makeatother

\makeatletter
\newsavebox\myboxA
\newsavebox\myboxB
\newlength\mylenA

\newcommand*\xtilde[2][0.8]{%
    \sbox{\myboxA}{$\m@th#2$}%
    \setbox\myboxB\null
    \ht\myboxB=\ht\myboxA%
    \dp\myboxB=\dp\myboxA%
    \wd\myboxB=#1\wd\myboxA
    \sbox\myboxB{$\m@th\widetilde{\copy\myboxB}$}
    \setlength\mylenA{\the\wd\myboxA}
    \addtolength\mylenA{-\the\wd\myboxB}%
    \ifdim\wd\myboxB<\wd\myboxA%
       \rlap{\hskip 0.5\mylenA\usebox\myboxB}{\usebox\myboxA}%
    \else
        \hskip -0.5\mylenA\rlap{\usebox\myboxA}{\hskip 0.5\mylenA\usebox\myboxB}%
    \fi}

\newbox\usefulbox

\def\getslant #1{\strip@pt\fontdimen1 #1}

\def\xxtilde #1{\mathchoice
 {{\setbox\usefulbox=\hbox{$\m@th\displaystyle #1$}%
    \dimen@ \getslant\the\textfont\symletters \ht\usefulbox
    \divide\dimen@ \tw@ 
    \kern\dimen@ 
    \xtilde{\kern-\dimen@ \box\usefulbox\kern\dimen@ }\kern-\dimen@ }}
 {{\setbox\usefulbox=\hbox{$\m@th\textstyle #1$}%
    \dimen@ \getslant\the\textfont\symletters \ht\usefulbox
    \divide\dimen@ \tw@ 
    \kern\dimen@ 
    \xtilde{\kern-\dimen@ \box\usefulbox\kern\dimen@ }\kern-\dimen@ }}
 {{\setbox\usefulbox=\hbox{$\m@th\scriptstyle #1$}%
    \dimen@ \getslant\the\scriptfont\symletters \ht\usefulbox
    \divide\dimen@ \tw@ 
    \kern\dimen@ 
    \xtilde{\kern-\dimen@ \box\usefulbox\kern\dimen@ }\kern-\dimen@ }}
 {{\setbox\usefulbox=\hbox{$\m@th\scriptscriptstyle #1$}%
    \dimen@ \getslant\the\scriptscriptfont\symletters \ht\usefulbox
    \divide\dimen@ \tw@ 
    \kern\dimen@ 
    \xtilde{\kern-\dimen@ \box\usefulbox\kern\dimen@ }\kern-\dimen@ }}%
 {}}

\newcommand*\xoverline[2][0.75]{%
    \sbox{\myboxA}{$\m@th#2$}%
    \setbox\myboxB\null
    \ht\myboxB=\ht\myboxA%
    \dp\myboxB=\dp\myboxA%
    \wd\myboxB=#1\wd\myboxA
    \sbox\myboxB{$\m@th\overline{\copy\myboxB}$}
    \setlength\mylenA{\the\wd\myboxA}
    \addtolength\mylenA{-\the\wd\myboxB}%
    \ifdim\wd\myboxB<\wd\myboxA%
       \rlap{\hskip 0.5\mylenA\usebox\myboxB}{\usebox\myboxA}%
    \else
        \hskip -0.5\mylenA\rlap{\usebox\myboxA}{\hskip 0.5\mylenA\usebox\myboxB}%
    \fi}

\def\xxoverline #1{\mathchoice
 {{\setbox\usefulbox=\hbox{$\m@th\displaystyle #1$}%
    \dimen@ \getslant\the\textfont\symletters \ht\usefulbox
    \divide\dimen@ \tw@ 
    \kern\dimen@ 
    \overline{\kern-\dimen@ \box\usefulbox\kern\dimen@ }\kern-\dimen@ }}
 {{\setbox\usefulbox=\hbox{$\m@th\textstyle #1$}%
    \dimen@ \getslant\the\textfont\symletters \ht\usefulbox
    \divide\dimen@ \tw@ 
    \kern\dimen@ 
    \xoverline{\kern-\dimen@ \box\usefulbox\kern\dimen@ }\kern-\dimen@ }}
 {{\setbox\usefulbox=\hbox{$\m@th\scriptstyle #1$}%
    \dimen@ \getslant\the\scriptfont\symletters \ht\usefulbox
    \divide\dimen@ \tw@ 
    \kern\dimen@ 
    \xoverline{\kern-\dimen@ \box\usefulbox\kern\dimen@ }\kern-\dimen@ }}
 {{\setbox\usefulbox=\hbox{$\m@th\scriptscriptstyle #1$}%
    \dimen@ \getslant\the\scriptscriptfont\symletters \ht\usefulbox
    \divide\dimen@ \tw@ 
    \kern\dimen@ 
    \xoverline{\kern-\dimen@ \box\usefulbox\kern\dimen@ }\kern-\dimen@ }}%
 {}}
\makeatother

\makeatletter
\newcommand{\mylabel}[2]{#2\def\@currentlabel{#2}\label{#1}}
\makeatother

\makeatletter
\newcommand{\Mac}{}
\DeclareRobustCommand{\Mac}{%
  M%
  \raisebox{\dimexpr\fontcharht\font`M-\height}{%
    \check@mathfonts\fontsize{\sf@size}{0}\selectfont
    c%
  }%
}
\makeatother

\newtheoremstyle{citing}
  {}
  {}
  {\itshape}
  {}
  {\bfseries}
  {\textbf{.}}
  {.5em}
  {\thmnote{#3}}

\theoremstyle{plain}

\newtheorem{theorem}[subsection]{Theorem}

\newtheorem{lemma}[subsection]{Lemma}
\newtheorem{corollary}[subsection]{Corollary}

\newtheorem{set}[subsection]{Setting}
\newtheorem{bigthm}{Theorem}

\newtheorem{proposition}[subsection]{Proposition}

\theoremstyle{remark}

\theoremstyle{definition}

\newtheorem{definition}[subsection]{Definition}

\numberwithin{equation}{section}

\theoremstyle{remark}
\newtheorem{remark}[subsection]{Remark}
\newtheorem*{claim}{Claim}
\newcounter{stepcounter}
\newtheorem{step}[stepcounter]{Step}

{\theoremstyle{citing}
}

{\theoremstyle{definition}
}

\title[Algebraic approximation and the decomposition theorem for K\"ahler CY varieties]{Algebraic approximation and the decomposition theorem\\ for K\"ahler Calabi--Yau varieties}

\author{Benjamin Bakker}
\address{Benjamin Bakker\\ Department of Mathematics\\ University of Illinois at Chicago\\ 851 S. Morgan St., Chicago, IL 60607}
\email{bakker.uic@gmail.com}

\author{Henri Guenancia}
\address{Henri Guenancia \\ Institut de Mathématiques de Toulouse; UMR 5219, Université de Toulouse; CNRS, UPS, 118 route de Narbonne, F-31062 Toulouse Cedex 9, France}
\email{henri.guenancia@math.cnrs.fr}

\author{Christian Lehn}
\address{Christian Lehn\\ Fakult\"at f\"ur Mathematik\\ Technische Universit\"at Chemnitz\\
Reichenhainer Stra\ss e 39, 09126 Chemnitz, Germany}
\email{christian.lehn@mathematik.tu-chemnitz.de}

\let\origmaketitle\maketitle
\def\maketitle{
  \begingroup
  \def\uppercasenonmath##1{} 
  \let\MakeUppercase\relax 
  \origmaketitle
  \endgroup
}

\begin{document}
\thispagestyle{empty}

\begin{abstract}
We extend the decomposition theorem for numerically $K$-trivial varieties with log terminal singularities to the K\"ahler setting.  Along the way we prove that all such varieties admit a strong locally trivial algebraic approximation, thus completing the numerically $K$-trivial case of a conjecture of Campana and Peternell.
\end{abstract}

\makeatletter
\@namedef{subjclassname@2020}{
	\textup{2020} Mathematics Subject Classification}
\makeatother


\maketitle

\setlength{\parindent}{1em}
\setcounter{tocdepth}{1}



\tableofcontents

\section{Introduction}\label{section introduction}
\thispagestyle{empty}
\def\qeX{\wt X}

For a compact Kähler manifold $X$ with vanishing first Chern class, the Beauville--Bogomolov theorem \cite{Bog78,Bea83} tells us that a (finite) étale cover of $X$ splits as a product of a complex torus, irreducible symplectic manifolds, and irreducible Calabi--Yau manifolds. Work of Druel--Greb--Guenancia--Höring--Kebekus--Peternell \cite{GKKP11,DG18, Dru18, GGK19, HP19} over the past decade has culminated in an analog of this theorem for projective varieties with log terminal singularities and numerically trivial canonical class, see \cite[Theorem 1.5]{HP19}. 

Our main result is a generalization of the decomposition theorem to the K\"ahler setting:
\begin{bigthm}\label{theorem decomposition}
Let $X$ be a numerically $K$-trivial compact Kähler variety with log terminal singularities. Then there is a quasi-étale cover $\qeX \to X$ such that $\qeX$ splits as a product
\[
\qeX=T \times \prod_i Y_i \times \prod_j Z_j
\]
where $T$ is a complex torus, the $Y_i$ are irreducible\footnote{Greb--Kebekus--Peternell use the term \emph{Calabi--Yau} for the irreducible factors in the decomposition of the second type, but it seems natural to call them \emph{irreducible Calabi--Yau}.} Calabi--Yau varieties, and the $Z_j$ are irreducible holomorphic symplectic varieties.
\end{bigthm}

A morphism $\qeX \to X$ of normal complex spaces is \emph{quasi-étale} if it is étale on the complement of an analytic subset which is locally of codimension at least $2$ in $\qeX$ and a \emph{cover} if it is finite and surjective. For convenience, we reproduce the definitions of irreducible Calabi--Yau and irreducible holomorphic symplectic varieties due to Greb--Kebekus--Peternell \cite[Definition~8.16]{GKP16sing} here.  Recall that if $X$ is a normal complex variety, the sheaf of reflexive $p$-forms $\Omega_{X}^{[p]}$ may be equivalently thought of as either the reflexive hull $(\Omega^p_{X})^{\vee \vee}$ or the push-forward $j_*\Omega^p_{X^\reg}$ from the regular locus $j:X^\reg\to X$.  If $X$ furthermore has rational singularities, it admits a third interpretation as $\pi_*\Omega^p_Y$ for any resolution $\pi: Y\to X$ by Kebekus--Schnell \cite[Corollary~1.8]{KS18}.

\begin{definition}\label{definition cy and isv}
Let $X$ be a compact Kähler variety with rational singularities. We call $X$ \emph{irreducible holomorphic symplectic} (IHS) if for all quasi-étale covers $q:\qeX \to X$, the algebra $H^0(\qeX, \Omega_{\qeX}^{[\bullet]})$ is generated by a holomorphic symplectic form $\wt \sigma \in H^0(\qeX,\Omega_{\qeX}^{[2]})$. We call $X$ \emph{irreducible Calabi--Yau} (ICY)  if for all quasi-étale covers $q:\qeX \to X$, the algebra $H^0(\qeX, \Omega_{\qeX}^{[\bullet]})$ is generated by a nowhere vanishing reflexive form in degree $\dim X$.
\end{definition}

Note that this is equivalent to the definition in \cite{GKP16sing}: in the presence of a reflexive form of degree $\dim X$ the singularities of $X$ are rational if and only if they are canonical by \cite[Théorème~1]{Elk81} and \cite[Corollary~1.8]{KS18}. 

The proof of the decomposition theorem in the projective case uses algebraic techniques (particularly regarding algebraic integrability of foliations, even though the usage of characteristic $p$ methods can be avoided by a recent paper of Campana \cite{Cam20}) which at the moment cannot be directly generalized to the analytic category. Instead, we reduce to the projective case via locally trivial deformations.  A crucial ingredient is therefore the following theorem, which resolves the numerically $K$-trivial case of a conjecture of Campana--Peternell\footnote{The conjecture has been attributed to Campana and Peternell in \cite{CHL19}. The authors are grateful to Thomas Peternell for bringing this conjecture to our attention.} saying that Kähler minimal models admit an algebraic approximation:

\begin{bigthm}\label{theorem peternell}
Any $X$ as in Theorem \ref{theorem decomposition} admits a strong locally trivial algebraic approximation:  there is a locally trivial family $\sX\to S$ over a smooth base $S$ specializing to $X$ over $s_0\in S$ such that points $s\in S$ for which $X_s$ is projective are analytically dense near $s_0$.\qed
\end{bigthm}

It is natural to ask whether the Bogomolov--Tian--Todorov theorem holds in this context---that is, whether locally trivial deformations of numerically $K$-trivial $X$ as in the theorem are always unobstructed (which would be sufficient to prove Theorem~\ref{theorem peternell}, see \cite[Theorem~1.2]{GS20}).  On the one hand, flat deformations of such $X$ are known to be potentially obstructed by an example of Gross \cite{Gro97}.  On the other hand, the proof of unobstructedness in the smooth case is fundamentally Hodge-theoretic, and from this perspective locally trivial deformations are more natural as they are topologically trivial (see \cite[Proposition~6.1]{AV19}).  While some special cases have been established (see \cite{BL16,BL18,GS20}), it is as yet unclear whether a locally trivial Bogomolov--Tian--Todorov theorem should hold.

The main difficulty in the proof of Theorem \ref{theorem peternell} is therefore to produce sufficiently many unobstructed deformations. Recall that by Kodaira's criterion, a compact Kähler manifold (or compact Kähler space with rational singularities) with no nonzero holomorphic $2$-forms is automatically projective, so the existence of 2-forms is the only obstruction to projectivity.  The results of \cite{GSS} extending the polystability of $T_X$ to the K\"ahler category provide a splitting of $T_X$ into foliations and the symplectic foliations among the factors (precisely, the weakly symplectic split foliations, see Definition~\ref{def good}) account for all of the 2-forms on $X$.  It is therefore natural to try to deform to the symplectic directions inside $H^1(X,T_X)$. We show that locally trivial deformations obtained by exponentiating the symplectic foliations of this splitting are always unobstructed.  As in the proof of the corresponding result for symplectic varieties \cite[Theorem~4.7]{BL18}, a crucial role is played by the degeneration of reflexive Hodge-to-de Rham in low degrees \cite[Section~2]{BL16}.  The general results we prove about locally trivial deformations along foliations (Section~\ref{section deformations}) and the existence of simultaneous resolutions in locally trivial families (Corollary~\ref{corollary resolutions}) are of independent interest as well.

With Theorem \ref{theorem peternell} in hand, the proof of Theorem \ref{theorem decomposition} proceeds as follows.  We first produce a locally trivial deformation $\pi:\sX\to \Delta$ of $X$ over the disk for which projective fibers are analytically dense.  By cycle-theoretic arguments and Theorem \ref{theorem decomposition} in the projective case \cite[Theorem 1.5]{HP19}, after replacing $X$ by a quasi-\'etale cover it suffices to assume there is a splitting $\sX^*=\sY^*\times_{\Delta^*} \sZ^*$ of the family $\sX^*:=\sX|_{\Delta^*}$ over the punctured disk, and we must show that the splitting extends over the puncture.  

One first observes that local triviality of the family $\pi:\sX\to\Delta$ implies the K\"unneth decomposition of $R^k\pi_*\Q_{\sX^*}$ extends, in fact as a decomposition of the variation of Hodge structures.  By $K$-triviality, the factors of the splitting of the tangent bundle $T_{\sX^*/\Delta^*}$ are cut out by differential forms and extend, so we have a splitting $T_{\sX/\Delta}\cong A\oplus B$.  The leaves of the splitting of the family over $\Delta^*$ have well-defined limits in the special fiber which are therefore tangent to the factors of the limit splitting $T_X=A_0\oplus B_0$ on the regular locus $X^\reg$.

It remains to show that the limit leaves define a product structure in the singular locus $X^\sing$.  There are essentially two types of phenomena that could go wrong:  the limit leaves could acquire new components in $X^\sing$, or limit leaves in the two directions could have positive-dimensional intersections in $X^\sing$.  To rule these out, we show that the splitting of the Ricci-flat metric $\omega_t=\omega_{A_t}+\omega_{B_t}$ of $X_t$ for $t\in\Delta^*$ induced by the splitting of the family extends over the puncture to a decomposition $\omega_0=\omega_{A_0}+\omega_{B_0}$ of the Ricci-flat metric on $X_0=X$ into closed semipositive currents with bounded potentials.  This is the key technical part of the proof of Theorem \ref{theorem decomposition} and the latter condition is critical:  it implies that these currents can be restricted to cycles in the singular locus and that one can compute intersection numbers with these currents. The fact that the decomposition is compatible with the limit K\"unneth decomposition and the semipositivity of the factors together imply neither pathology arises.

\medskip

Patrick Graf informed us that he independently obtained a Kähler version of the decomposition theorem in dimension at most four, see \cite{Gra21}.

\subsection{Outline}

\begin{itemize}
\item[$\bullet$] Section \ref{section deformations}. We collect some background on locally trivial deformations, define locally trivial deformations along foliations, and prove unobstructedness of deformations along weakly symplectic split foliations.  
\item[$\bullet$] Section \ref{section ktrivial}. We recall the precise notions of $K$-triviality and prove Theorem \ref{theorem peternell}. We derive some first applications about fundamental groups and deformation of the irreducible building blocks. 
\item[$\bullet$] Section \ref{section douady}. We recall some foundational aspects of relative Douady spaces and show that local product decompositions can be spread out over Zariski open sets.  
\item[$\bullet$] Section~\ref{section tangent splitting}. A locally trivial family $\sX\to\Delta$ which is a product over $\Delta^*$ admits a limit product decomposition on cohomology. We deduce from that a splitting for the relative tangent sheaf.
\item[$\bullet$] Section \ref{section kaehler einstein}. We prove that the K\"ahler--Einstein metric in the limit splits as a sum of positive currents with bounded local potentials corresponding to the product decomposition of the family $\sX\to \Delta$ over $\Delta^*$.
\item[$\bullet$] Section \ref{section family splitting}. Building upon the previous results, we prove a global splitting result for locally trivial families $\sX\to\Delta$ that are a product over $\Delta^*$, under some additional conditions.
\item[$\bullet$]  Section \ref{section decomp proof}. We proceed to checking that the assumptions in the splitting theorem from the previous section are fulfilled in our geometric setting, thereby proving Theorem~\ref{theorem decomposition}.
\end{itemize}
\subsection*{Acknowledgments.} 
We benefited from discussions, remarks, and emails of Benoît Claudon, Stéphane Druel, Patrick Graf, Vincent Guedj, Stefan Kebekus, Mihai P\u{a}un, Thomas Peternell, Christian Schnell, Bernd Schober, and Ahmed Zeriahi. The authors are grateful to the referee for the various suggestions that have greatly improved the exposition.

Benjamin Bakker was partially supported by NSF grant DMS-1848049. Henri Guenancia has benefited from State aid managed by the ANR under the "PIA" program bearing the reference ANR-11-LABX-0040,
 in connection with the research project HERMETIC. Christian Lehn was supported by the DFG through the research grants Le 3093/2-2 and  Le 3093/3-1.

\subsection*{Notation and Conventions} 
A resolution of singularities of a variety $X$ is a proper surjective bimeromorphic morphism $\pi:Y \to X$ from a nonsingular variety $Y$. The term variety will denote an integral separated scheme of finite type over $\C$ in the algebraic setting or an irreducible and reduced separated complex space in the complex analytic setting. For a field $k$, an algebraic $k$-scheme is a scheme of finite type over $k$. We will denote by $\Delta:=\{z\in \mathbb C \mid |z|<1\}$ the complex unit disk and by $\Delta^*:=\Delta\setminus\{0\}$ the punctured disk. 


\section{Locally trivial deformations along foliations and resolutions}\label{section deformations}

Throughout we define $\Art_k$ to be the category of local artinian $k$-algebras. To simplify the notation, we agree that $k$ will denote an algebraically closed field when speaking about schemes and $k=\C$ when speaking about complex spaces.
\subsection{Locally trivial deformations}\label{section locally trivial}
We begin with some background on locally trivial deformations.

\begin{definition}
Let $f:\sX\to S$ be a morphism of complex spaces (or algebraic schemes\footnote{or algebraic spaces, if we use the \'etale topology in the sequel.}).  We say:
\begin{enumerate}
\item $\sX$ is \emph{locally trivial} over $S$ if there is a cover $\sX_i$ of $\sX$, a cover $S_i$ of $S$ such that $\sX_i\to S$ factors through $S_i$, complex spaces (or schemes) $X_i$, and commutative diagrams
\[\xymatrix{
\sX_i\ar[d]\ar[r]^\cong&X_i\times S_i\ar[ld]\\
S_i}\]
where the diagonal map is the projection.
\item  $\sX$ is \emph{formally locally trivial} over $S$ if for any $T=\Spec A\to S$ with $A\in \Art_k$ the base-change $\sX_T\to T$ is locally trivial.
\end{enumerate}
\end{definition}

\begin{remark}  \hspace{.5in}
\begin{enumerate}
\item Of course, over an artinian base, the notions of formal local triviality and local triviality are equivalent. 
\item In the analytic category, by results of Artin \cite[Theorem (1.5)(ii)]{Art68}, $\sX/S$ is locally trivial if and only if $\wh\sO_{\sX,x}\cong \wh\sO_{X_s,x}\otimes_\C\wh\sO_{S,f(x)} $ as $\wh\sO_{S,f(x)}$-algebras for all points $x\in\sX$.  Thus, in the analytic category $\sX/S$ is locally trivial if and only if it is formally locally trivial.  In the algebraic category, local triviality is in general much stronger than formal local triviality over nonartinian bases. 
\end{enumerate}
\end{remark}

\begin{definition}\label{definition flt}
Let $X/k$ be a complex space (or an algebraic scheme). The locally trivial deformation functor $F^\lt_X:\Art_k\to\mathrm{Sets}$ is defined as follows: $F^\lt_X(A)$ is the set isomorphism classes of locally trivial families $\sX/\Spec A$ together with a $k$-morphism $X\to \sX$ which is an isomorphism modulo $\mathfrak{m}_A$. Here, we consider isomorphism classes for isomorphisms which are the identity modulo $\mathfrak{m}_A$.
\end{definition}

We recall that locally trivial deformations are controlled by the tangent sheaf $T_{\sX/S}:=\sHom_{\sO_\sX}(\Omega^1_{\sX/S},\sO_\sX)$.  This will be made precise in a way that can be adapted easily to deformations preserving a foliation in Section~\ref{section locally trivial foliation}.  For $A$ in $\Art_k$ let $$\mathbf{G}_X(A):=\Aut_A(\sO_X\otimes_k A)$$ be the sheaf of $A$-algebra automorphisms of $\sO_X\otimes_k A$, and let $\mathbf{U}_X(A)\subset\mathbf{G}_X(A)$ be the subsheaf of automorphisms which are the identity modulo $ \mathfrak{m}_A$.

The following proposition is immediate:

\begin{proposition}\label{proposition defo}
Let $X/k$ be a complex space (or an algebraic scheme) and $F^\lt_X$ its locally trivial deformation functor.  Then we have a natural identification:
\[F^\lt_X(A)=\check H^1(X,\mathbf{U}_X(A)).\]
\end{proposition}

Note that in characteristic zero we have an isomorphism of sheaves of pointed sets
\[
\exp: T_X\otimes_k \mathfrak{m}_A\to \mathbf{U}_X(A)
\] 
where $T_X\otimes_k A$ is given the obvious structure of a sheaf of $A$-linear Lie algebras.  For any small extension\footnote{Recall that a sequence such as \eqref{eq:smallext} is a \emph{small extension} if $J.\gothm_{A'}=0$. In this case, the $A'$-module structure of $J$ factors through $k=A'/\gothm_{A'}=A/\gothm_{A}$.}
\begin{equation}\label{eq:smallext}
0\to J\to A'\to A\to 0
\end{equation}
with $A,A'\in\Art_k$ the first row of the following commutative diagram is then exact:
\begin{equation}\label{eq lift auto}
\begin{tikzcd}
0\ar[r]&T_X\otimes_k J\ar[r]\ar[d,equal]&T_X\otimes_k \mathfrak{m}_{A'}\ar[r]\ar{d}{\exp}&T_X\otimes_k\mathfrak{m}_A\ar[r]\ar{d}{\exp}&0\\
0\ar[r]& T_X\otimes_k J\ar{r}{\exp}& \mathbf{U}_X(A')\ar[r]& \mathbf{U}_X(A)\ar[r] &1.
\end{tikzcd}
\end{equation}
Here the horizontal maps are morphisms of sheaves of groups and the right and center vertical maps are isomorphisms of sheaves of pointed sets.

It follows that the bottom row is exact. Moreover, as the top row is exact on global sections, it follows that the bottom row is exact on global sections as well.

\begin{corollary}\label{cor lt def thry}
Suppose $X/k$ is a separated complex space (or a separated algebraic scheme with $\mathrm{char}(k)=0$).  Then the following hold.
\begin{enumerate}
\item 
The functor $F_X^\lt$ admits a tangent-obstruction theory with tangent space equal to $H^1(X,T_X)$ and obstructions in $H^2(X,T_X)$.
\item\label{item def mod} For any family $\sX/ S=\Spec A$ in $F^\lt_{X}(A)$, the lifts of $\sX/S$ to $F^\lt_X(A[\epsilon])$ are canonically parametrized by a functorial quotient of $H^1(\sX,T_{\sX/S})$.
\end{enumerate}
\end{corollary}

\begin{remark}\label{remark H5}It follows that $F^\lt_X$ satisfies Schlessinger's axioms $(H_1)$-$(H_3)$, see \cite[Theorem~2.11]{Sch68}.  Note that while $F^\lt_X$ may not satisfy $(H_4)$, it does satisfy axiom $(H_5)$ of \cite[\S~1]{Gro97b}, since the fibered coproduct of two deformations may be constructed by taking the fibered direct product of the sheaves of rings.  Thus, the deformation module $T^1(\sX/S)$ has the structure of an $A$-module, and in part \eqref{item def mod} we mean that it is a quotient of $H^1(\sX,T_{\sX/S})$ as an $A$-module which is compatible with restriction maps.  If $\sX/S$ has no automorphisms restricting to the identity on the special fiber, then $H^1(\sX,T_{\sX/S})\to T^1(\sX/S)$ will be an isomorphism. These remarks likewise hold for the other deformation functors that will be defined in Section \ref{section locally trivial foliation}.  
\end{remark}

\begin{proof}[Proof of Corollary \ref{cor lt def thry}]
The following lemma describes how much of the long exact sequence survives in the cohomology of $\mathbf{U}_X(A)$:
\begin{lemma}\label{lemma group sheaf cohomology}
Let 
\begin{equation}0\to T\to G'\to G\to 1\label{eq ses}\end{equation}
be an exact sequence of sheaves of groups on a topological space where $T$ is abelian.
\begin{enumerate}
\item If $T$ is central in $G$, then we have a sequence
\[
\begin{tikzcd}
\check H^1(X,T)\arrow{r}&\check H^1(X,G')\arrow{r}&\check H^1(X,G)\arrow{r}{\delta}&\check H^2(X,T)
\end{tikzcd}
\]
where
\begin{enumerate}
\item The natural action of $\check H^1(X,T)$ on $\check H^1(X,G')$ is transitive on fibers;
\item The image of $\check H^1(X,G')\to \check H^1(X,G)$ is the inverse image of $0$ under $\delta$.
\end{enumerate}
\item If \eqref{eq ses} is split exact\footnote{Meaning $G'\to G$ has a section.}, then for each $\alpha\in \check H^1(X,G)$ the natural action of $\check H^1(X,T^\alpha)$ on the fiber of the map
\[
\begin{tikzcd}
\check H^1(X,G')\arrow{r}&\check H^1(X,G)
\end{tikzcd}
\]
above $\alpha$ is transitive, where $T^\alpha$ is the sheaf obtained from $T$ by twisting by $\alpha$.
\end{enumerate}
\end{lemma}
\begin{proof}Easily checked with \v{C}ech cochains.
\end{proof}
Now, the first claim is immediate upon taking the long exact sequence on \v Cech cohomology of the second row of \eqref{eq lift auto} using the first part of the lemma (where we used separatedness to identify \v Cech cohomology with sheaf cohomology). For the second part, we have a split exact sequence
\begin{equation}\label{eq lift auto two}
\begin{tikzcd}
0\ar[r]& T_X\otimes_k A\ar{r}{\exp}& \mathbf{U}_X(A[\epsilon])\ar[r]& \mathbf{U}_X(A)\ar[r] &1
\end{tikzcd}
\end{equation}
and the claim follows from the second part of the lemma, as the stabilizer of an element $(\sX'/S')$ under the action of $H^1(\sX,T_{\sX/S})$ is easily seen to be an $A$-submodule.  Note that for $\alpha=(\sX/S)\in \check H^1(X,\mathbf{U}_X(A))$ we naturally have $(T_{X\times S/S})^\alpha=T_{\sX/S}$.	
\end{proof}

\begin{remark} We would like to make a couple of remarks regarding \eqref{eq lift auto}.
\begin{enumerate}
\item The restriction morphism $\mathbf{U}_X(A')\to \mathbf{U}_X(A)$ may fail to be surjective in characteristic $p$. If we take $X=\Spec k[x]/(x^p)$, $A=k[\epsilon]/(\epsilon^p)$, and $A'=k[\epsilon]/(\epsilon^{p+1})$, the automorphism $x\mapsto x+\epsilon$ of $X\times\Spec A$ does not lift.
\item Let  $A' \to A$ be a small extension in $\Art_k$, let $\sX'\to \Spec A'$ be flat, and $\sX:=\sX'\times_{\Spec A'}\Spec A$. The same argument shows that $\Aut_{A'}(\sX')\to\Aut_{A}(\sX)$ is surjective whenever $T_{\sX'/A'} \to T_{\sX/A}$ is. Example 2.6.8(i) in \cite{Ser06} shows that for $A'=k[t]/t^3 \to k[t]/t^2 = A$ and $\sX'=\Spec\left(k[x,y,t]/(xy-t,t^3)\right)$ the automorphism of $\sX$ determined by $x\mapsto x+tx$ and $y\mapsto y$ does not lift to $\sX'$. But neither does the vector field $t\dfrac{\di}{\di x} \in T_{\sX/A}$.
\end{enumerate}
\end{remark}

\subsection{Locally trivial deformations along foliations}\label{section locally trivial foliation}
The above results now easily extend to the situation of deformations along a foliation, provided we require \mbox{$\mathrm{char}(k)=0$} in the algebraic case.  For simplicity, we therefore take $k=\C$ in the remainder.

\begin{definition}
Let $X$ be a separated complex space (or a separated complex algebraic scheme). The {\it tangent sheaf} of $X$ is defined by $T_X:=\sHom_{\sO_X}(\Omega^1_{X},\sO_X)$. A {\it foliation} is a subsheaf $E\subset T_X$, such that $E$ is stable under the Lie bracket and saturated in $T_X$, i.e. such that the quotient $T_X/E$ is torsion-free.
\end{definition}
 
\begin{definition}\label{definition flte}
Let $X$ be a separated complex space (or a separated complex algebraic scheme) with a foliation $E\subset T_X$.  For $A\in\Art_\C$ set $\mathbf{U}_E(A):=\exp(E\otimes_\C\mathfrak{m}_A)\subset\mathbf{U}_X(A)$. We define the deformation functor $F_E^\lt:\Art_\C\to\mathrm{Sets}$ by
\[
F_E^\lt(A):=\check H^1(X,\mathbf{U}_E(A)).
\]
\end{definition}

We somewhat abusively refer to sections of $F^\lt_E$ as $(\sX/S)\in F^\lt_E(A)$ even though the natural map $F^\lt_E\to F^\lt_X$ may not be injective on sections. Note that the integrability of $E$ is necessary in order for $\mathbf{U}_E(A)\subset \mathbf{U}_X(A)$ to be a sheaf of subgroups.

\begin{proposition}\label{proposition defo foliation}
Let $X$ be a separated complex space (or a separated complex algebraic scheme) with a foliation $E\subset T_X$. Then the following hold.
\begin{enumerate}
\item The functor $F_E^\lt$ admits a tangent-obstruction theory with tangent space $H^1(X,E)$ and obstructions in $H^2(X,E)$.
\item Associated to any family $\sX/ S=\Spec A$ in $F^\lt_{E}(A)$ there is a functorial foliation $E_{\sX/S}\subset T_{\sX/S}$ which locally agrees with the trivial extension of $E$ on any local trivialization of the $\mathbf{U}_E(A)$-cocycle representing $\sX/S$.  The lifts of $\sX/S$ to $F^\lt_E(A[\epsilon])$ are canonically parametrized by a functorial quotient of $H^1(\sX,E_{\sX/S})$.
\end{enumerate}
\end{proposition}

\begin{remark}
Note that the functor $F^\lt_E$ is not the functor of locally trivial deformations for which $E$ lifts locally trivially:  there may well be sections of $\mathbf{U}_X(A)$ which stabilize $E\otimes_\C A$ but do not come from exponentiating $E\tensor_\C \gothm_A$.  Indeed, take $X=Y\times Z$ with the induced splitting 
\[T_X=\pi_1^*T_Y\oplus \pi_2^*T_Z\]
and $E=\pi_1^*T_Y$.  Then any locally trivial deformation of the two factors will obviously yield a locally trivial deformation of $X$ for which the two foliations lift locally trivially, but such a deformation does not in general come from a section of $F^\lt_E$.  Moreover, in this case, the gluing maps for a section of $F^\lt_E$ are not required to preserve $\pi_2^*T_Z$---that is, they are not required to be constant in the $Z$ direction.  
\end{remark}

\subsection{Locally trivial deformations of tangent splittings}

Throughout we say that a subsheaf $Q\subset E$ is locally split if it is locally a direct summand.

\begin{definition}
Let $X$ be a separated complex space (or a separated complex algebraic scheme), and suppose we have $Q\subset E\subset T_X$ where $E$ is a foliation and $Q$ is a locally split subsheaf of $E$.  For $A\in \Art_\C$ and $S:=\Spec A$ we define $F^\lt_{E;Q}(A)$ to be the set of $(\sX/S)\in F^\lt_E(A)$ together with a flat lift $Q_{\sX/S}\subset E_{\sX/S}$ of $Q$, up to the obvious notion of isomorphism.  We usually write $(\sX/S)\in F^\lt_{E;Q}(A)$ when we mean $(\sX/S,Q_{\sX/S})\in F^\lt_{E;Q}(A)$.  For $E=T_X$ itself, we denote $F^\lt_{X;Q}:=F^\lt_{T_X;Q}$. 
\end{definition}

Note that choices of a split complement to $B_{\sX/S}\subset E_{\sX/S}$ naturally form a pseudotorsor\footnote{That is, a torsor if nonempty.} for $\Hom_{\sO_\sX}(E_{\sX/S}/B_{\sX/S},B_{\sX/S})$.  Thus, on a local trivialization of the $\mathbf{U}_E(A)$-cocycle representing $\sX/S$ over a cover $U_i$ where we have a splitting $E|_{U_i}=Q|_{U_i}\oplus B_i$, the sheaf $Q_{\sX/S}$ is locally identified with $(1+f_i)(Q|_{U_i}\otimes_\C A)\subset E|_{U_i}\otimes_\C A$ for some $f_i\in \Hom_{\sO_{U_i}}(Q|_{U_i},B_i)\otimes_\C \mathfrak{m}_A$.

Given a locally trivial deformation $\sX/S$ over an artinian base $S$, the relative tangent sheaf $T_{\sX/S}$ acts via the adjoint representation on $T_{\sX/S}$, and for $(\sX/S)\in F^\lt_{E;Q}(A)$ and any local section $e$ of $E_{\sX/S}$ we obtain an $\sO_\sX$-linear map $\mathrm{ad}^{E;Q}(e):Q_{\sX/S}\to E_{\sX/S}/Q_{\sX/S}$.  We define a two-term complex
\[M^{E;Q}_{\sX/S}:=[E_{\sX/S}\xrightarrow{\mathrm{ad}^{E;Q}}\sHom_{\sO_\sX}(Q_{\sX/S},E_{\sX/S}/Q_{\sX/S})]\]
supported in degrees $[0,1]$.  For simplicity we write $M^{E;Q}:=M^{E;Q}_{X/\C}$.

\begin{proposition}\label{proposition defo sub}
Let $X$ be a separated complex space (or a separated complex algebraic scheme) with $Q\subset E\subset T_X$ where $E$ is a foliation and $Q$ is a locally split subsheaf of $E$.  Then the following hold.  
\begin{enumerate}
\item 
The functor $F_{E;Q}^\lt$ admits a tangent-obstruction theory with tangent space equal to $\HH^1(X,M^{E;Q})$ and obstructions in $\HH^2(X,M^{E;Q})$.
\item For any $(\sX/ S)\in F^\lt_{E;Q}(A)$ the lifts of $\sX/S$ to $F^\lt_{E;Q}(A[\epsilon])$ are canonically parametrized by a functorial quotient of  $\HH^1(\sX,M^{E;Q}_{\sX/S})$. 
\end{enumerate}
\end{proposition}

Before the proof it will be useful to explicitly describe the \v{C}ech hypercohomology of two-term complexes.  For a two-term complex $K=[A\xrightarrow{f} B]$ of sheaves on a topological space supported in degrees $[0,1]$ and a cover $\underline{U}=\{U_i\}$, by taking the total \v{C}ech complex we see that the \v{C}ech cochains and coboundary operators are given by 
\begin{align*}
C^\ell(\underline{U},K)&=C^\ell(\underline{U},A)\oplus C^{\ell-1}(\underline{U},B)\\
\delta(a,b)&=(\delta a,\delta b+(-1)^{\deg a}f(a)).
\end{align*}
We write
\begin{align*}
Z^\ell(\underline{U},K)&:=\ker\left(C^\ell(\underline{U},K)\xrightarrow{\delta} C^{\ell+1}(\underline{U},K)\right)\\
B^\ell(\underline{U},K)&:=\img\left(C^{\ell-1}(\underline{U},K)\xrightarrow{\delta} C^\ell(\underline{U},K)\right)
\end{align*}
for the $\ell$-cocycles and $\ell$-coboundaries.

\begin{proof}[Proof of Proposition \ref{proposition defo sub}]  Both parts are easily seen via \v{C}ech cohomology.  By Proposition \ref{proposition defo foliation}, any $(\sX/S)\in F^\lt_E(A)$ is trivialized on a Stein (or affine) open cover $\underline{U}:=\{U_i\}$ of $X$, and we may further assume we have $E|_{U_i}=Q|_{U_i}\oplus B_i$ for some sheaf $B_i$.  As nilpotent thickenings of Stein spaces are Stein, see e.g. \cite[Kapitel~V, \S~4.3, Satz~5]{GR77}, we may compute (hyper)cohomology in the following using \v{C}ech cohomology with the cover $\{U_i\}$.

For the tangent space claim in the first part, take a small extension $A'\to A$ with kernel $J$, and assume $(\sX'/S'),(\sX'',S')\in F^\lt_{E;Q}(A')$ both lift $(\sX/S)\in F^\lt_{E;Q}(A)$.  If $(\sX'/S')\in F^\lt_E(A')$ is given by gluing data $g'_{ij}$ on $U_{ij}$, then $\sX''/S'$ is given by gluing data $g'_{ij}(1-e_{ij})$ for a 1-cocycle $e$ valued in $E\otimes_\C J$, by Proposition \ref{proposition defo foliation}.  With respect to those local trivializations we have $Q_{\sX'/S'}=(1+f'_i)(Q\otimes_\C A')$, and therefore $Q_{\sX''/S}=(1+f'_i+(\pi_i\otimes 1)^{-1}v_i)(Q\otimes_\C A')$ for a 0-cochain $v$ valued in $\sHom_{\sO_X}(Q,E/Q)\otimes_\C J$, where $\pi_i:B_i\to(E/Q)|_{U_i}$ is the obvious isomorphism.  Now for the $Q_{\sX'/S'}$ to glue we must have that $(1-f'_j)g'_{ij}(1+f'_i)$ preserves $Q$, and likewise for the $Q_{\sX''/S'}$ to glue we must have that
\begin{align*}&(1-f'_j-(\pi_j\otimes 1)^{-1}v_j)g'_{ij}(1-e_{ij})(1+f'_i+(\pi_i\otimes 1)^{-1}v_i)\\
&\hspace{1.5in}=(1-f'_j)g'_{ij}(1+f'_i)+(-e_{ij}+(\pi_i\otimes 1)^{-1}v_i-(\pi_j\otimes 1)^{-1}v_j)\end{align*}
preserves $Q$, and therefore that $\mathrm{ad}^{E;Q}(e)=\delta v$.  Working backward, $Z^1(\underline{U},M^{E;Q})\otimes_\C J$ naturally acts transitively on the set of lifts of $(\sX/S)\in F^\lt_{E;Q}(A)$, and the 1-coboundaries are easily seen to act trivially.

For the obstruction space claim in the first part, take $(\sX/S)\in F^\lt_{E;Q}(A)$ with gluing data $g_{ij}$ and such that $Q_{\sX/S}$ is locally identified with $(1+f_i)(Q\otimes_\C A)$.  Choose arbitrary lifts $g'_{ij}$ of $g_{ij}$ and $f_i'$ of $f_i$.  Then taking $1+e_{ijk}=g_{ik}^{\prime-1}g_{jk}'g_{ij}'$ and $-v_{ij}\in H^0(U_{ij},\sHom_{\sO_X}(Q,E/Q)\otimes_\C J)$ the map induced by $(1-f_j')g_{ij}'(1+f_i')$, the 2-cochain $(e,v)$ is easily seen to be a 2-cocycle for $M^{E;Q}\otimes_\C J$ and to have cohomology class in $\HH^2(X,M^{E;Q})\otimes_\C J$ which is independent of the choices.  If it is a coboundary, say $(e,v)=\delta(x,y)$, then $g'_{ij}(1-x_{ij})$ satisfies the cocycle condition and thus gives gluing data for a lift $(\sX'/S')\in F^\lt_E(A')$, and further $v_{ij}=-x_{ij}+y_i-y_j$, so the $(1+f_i+(\pi_i\otimes 1)^{-1}y_i)(Q\otimes_\C A)$ glue.

The second part follows by the same sort of computation as the proof of the tangent space claim in the first part, and is left to the reader.
\end{proof}
\def\split{\mathrm{split}}
\begin{definition}
Let $X$ be a separated complex space (or a separated complex algebraic scheme) with a splitting $T_X=\oplus_i E^{(i)}\oplus P$ of the tangent sheaf where $E:=\oplus_i E^{(i)}$ is a foliation.  For $A\in \Art_\C$ and $S:=\Spec A$ we define $F^\lt_{E;\mathrm{split}}(A)$ to be the set of $(\sX/S)\in F^\lt_E(A)$ together with a lift $T_{\sX/S}=\oplus_iE^{(i)}_{\sX/S}\oplus P_{\sX/S}$ of the splitting for which $E_{\sX/S}=\oplus_i E^{(i)}_{\sX/S}$, up to the obvious notion of isomorphism.  We suppress the specification of the splitting from the notation for $F^\lt_{E;\split}$ as it will be clear from context, and we once again abusively write $(\sX/S)\in F^\lt_{E;\split}(A)$ when we mean $(\sX/S,E^{(1)}_{\sX/S},\ldots,E^{(n)}_{\sX/S},P_{\sX/S})\in F^\lt_{E;\split}(A)$ where $n$ is the number of summands of $E$. 
\end{definition}
 
 We have $$F^\lt_{E;\split}=\left(F^\lt_{E;E^{(1)}}\times_{F^\lt_X}\cdots\times_{F^\lt_X}F^\lt_{E;E^{(n)}}\right)\times_{F^\lt_X}F^\lt_{X;P}.$$  The complex 
{\small\[M^{E;\split}_{\sX/S}:=\left[E_{\sX/S}\xrightarrow{\oplus_i\mathrm{ad}^{E;E^{(i)}}\oplus\mathrm{ad}^{X;P}}\bigoplus_i\sHom_{\sO_\sX}(E^{(i)}_{\sX/S},E_{\sX/S}/E^{(i)}_{\sX/S})\oplus \sHom_{\sO_\sX}(P_{\sX/S},T_{\sX/S}/P_{\sX/S})\right]\]}
is easily seen to be 
\[M^{E;\split}_{\sX/S}=\left(M^{E^{(1)}}_{\sX/S}\oplus_{E_{\sX/S}}\cdots\oplus_{E_{\sX/S}}M^{E^{(n)}}_{\sX/S}\right)\oplus_{T_{\sX/S}}M^P_{\sX/S}\]
and it is straightforward to conclude:
\begin{corollary}\label{corollary defo split}
Let $X$ be a separated complex space (or a separated complex algebraic scheme) with a splitting $T_X=\oplus_i E^{(i)}\oplus P$ of the tangent sheaf where $E:=\oplus_i E^{(i)}$ is a foliation. Then the following hold.
\begin{enumerate}
\item 
The functor $F_{E;\split}^\lt$ admits a tangent-obstruction theory with tangent space equal to $\HH^1(X,M^{E;\split})$ and obstructions in $\HH^2(X,M^{E;\split})$.
\item For any $(\sX/ S)\in F^\lt_{E;\split}(A)$ the lifts of $\sX/S$ to $F^\lt_{E;\split}(A[\epsilon])$ are canonically parametrized by a functorial quotient of  $\HH^1(\sX,M^{E;\split}_{\sX/S})$. 
\end{enumerate}
\end{corollary}

By Schlessinger's criterion \cite[Theorem~2.11]{Sch68}, Proposition~\ref{proposition defo}, Proposition~\ref{proposition defo foliation}, and Corollary \ref{corollary defo split}, when $X/\C$ is proper, the functors $F_X^\lt$, $F_E^\lt$, and $F^\lt_{E;\split}$ all admit miniversal formal families in the category of formal complex spaces (or formal algebraic schemes), and we denote by $\wh\Def^\lt(X)$, $\wh{\Def}_E^\lt(X)$, and $\wh{\Def}_{E;\split}^\lt(X)$ the bases of such a miniversal formal family, which is unique up to (not necessarily unique) isomorphism.

\begin{corollary}\label{cor def fol}  In the setup of Proposition~\ref{proposition defo sub}, assume further that $X$ is proper.  Then there are maps $$\wh{\Def}_{E;\split}^\lt(X)\to\wh{\Def}_{E}^\lt(X)\to \wh\Def^\lt(X)$$ of formal spaces whose derivatives are the natural maps $$\HH^1(X,M^{E;\split})\to H^1(X,E)\to H^1(X,T_X).$$\end{corollary}

\subsection{Kuranishi spaces for locally trivial deformations}\label{section kuranishi}

We recall some results in the analytic category realizing formal deformation-theoretic objects as completions of germs.
\begin{theorem}[Grauert, Douady]\label{theorem grauert douady}
For any compact complex space $Z$ there exists a miniversal deformation $\scrZ \to \Def(Z)$ over a germ $\Def(Z)$ which is a versal deformation of all of its fibers.
\end{theorem}
\begin{proof}
This is \cite[Hauptsatz, p 140]{Gra74}, see also \cite[Th\'eor\`eme principal, p 598]{Dou74}. 
\end{proof}
The family $\scrZ\to \Def(Z)$ is called the \emph{Kuranishi family} and $\Def(Z)$ is called \emph{Kuranishi space}. If $Z$ is a complex space satisfying $H^0(Z,T_Z)=0$, then every miniversal deformation is universal.  

We recall the analog of Theorem~\ref{theorem grauert douady} for locally trivial deformations.

\begin{theorem}[Flenner--Kosarew]\label{theorem flenner kosarew}
For a miniversal deformation $\scrZ \to \Def(Z)$ of a compact complex space $Z$ there exists a closed complex subspace $\Def^\lt(Z)\subset \Def(Z)$ of the Kuranishi space such that 
\[
\scrZ \times_{\Def(Z)}\Def^\lt(Z) \to \Def^\lt(Z)
\]
is a locally trivial deformation of $Z$ and is miniversal for locally trivial deformations of $Z$.
\end{theorem}
\begin{proof}
This is \cite[(0.3) Corollary]{FK87}.
\end{proof}

\subsection{Locally trivial resolutions}
\begin{definition}
Let $S,\sX,\sY$ be complex spaces (or complex algebraic schemes), $\sX\to S$ and $\sY\to S$ morphisms, and $f:\sY\to \sX$ an $S$-morphism.  We say:
\begin{enumerate}
\item $f$ is \emph{locally trivial} over $S$ if there is a cover $\sX_i$ of $\sX$, a cover $S_i$ of $S$, and morphisms $g_i:Y_i\to X_i$ together with diagrams (over $S$) 
\[\xymatrix{
\sY_i\ar[d]_{f|_{\sY_i}}\ar[r]^\cong&Y_i\times S_i\ar[d]^{g_i\times\id}\\
\sX_i\ar[r]^\cong&X_i\times S_i
}\] 
where $\sY_i=f^{-1}(\sX_i)$.  
\item  $f$ is \emph{formally locally trivial} over $S$ if for any $T=\Spec A\to S$ with $A\in \Art_\C$ the base-change $f_T:\sY_T\to\sX_T$ is locally trivial over $T$.
\end{enumerate}
If $f$ is (formally) locally trivial and fiberwise a resolution, we say it is a \emph{(formally) locally trivial resolution (over $S$)}.
\end{definition}

Let $X,Y$ be separated complex spaces (or separated complex algebraic schemes) and $\pi:Y\to X$ a $\C$-morphism.  There is a naturally defined deformation functor $F^\lt_{Y/X}:\Art_\C\to \mathrm{Sets}$ of locally trivial deformations $\sX/S$ and $\sY/S$ of $X$ and $Y$, respectively, together with a locally trivial deformation $\sY\to\sX$ of $\pi$.  Let $\mathbf{U}_{Y/X}(A)\subset \pi_*\mathbf{U}_Y(A)\times \mathbf{U}_X(A)$ be the sheaf of subgroups whose sections over $U\subset X$ are pairs of $A$-automorphisms $(f,g)$ making the following square commute over $\Spec A$
\[\begin{tikzcd}
Y_U\times \Spec A\arrow{r}{f}\arrow{d}&Y_U\times \Spec A\arrow{d}\\
U\times \Spec A\arrow{r}{g}&U\times \Spec A \end{tikzcd}
\]
where $Y_U=\pi^{-1}(U)$.  We have a natural identification
\[F^\lt_{Y/X}(A)=\check H^1(X,\mathbf{U}_{Y/X}(A)).\]
There is also a natural map of functors $F^\lt_{Y/X}\to F^\lt_X$ coming from the projection map $\mathbf{U}_{Y/X}(A)\to\mathbf{U}_X(A)$.

For the next proposition, note that for any morphism $\pi:Y\to X$ we have a natural map
\[\pi_*\mathrm{Der}_\C(\sO_Y,\sO_Y)\to\mathrm{Der}_\C(\pi_*\sO_Y,\pi_*\sO_Y).\] 
In particular, if we have $\pi_*\sO_Y=\sO_X$ (via the natural map), then there is a natural map $\pi_*T_Y\to T_X$.
\begin{proposition}\label{prop morph deforms}  Let $X$ be a separated complex space (or a separated complex algebraic scheme).  Let $\pi:Y\to X$ be a morphism for which $\pi_*\sO_Y=\sO_X$ and such that the induced map $\pi_*T_Y\to T_X$ is an isomorphism.  Then the natural map $F^\lt_{Y/X}\to F^\lt_X$ is an isomorphism.
\end{proposition}

\begin{proof}
It is enough to show that for all $A\in\Art_\C$ the natural map $\mathbf{U}_{Y/X}(A)\to\mathbf{U}_X(A)$ is an isomorphism, and this is obvious by induction on small extensions using the hypothesis on the tangent sheaves and the analog of \eqref{eq lift auto} for $\mathbf{U}_{Y/X}$.
\end{proof}

By \cite[Corollary~4.7]{GKK10}, for a reduced and normal complex space $X$, a log resolution $\pi:Y\to X$ for which $\pi_{*}T_{Y}=T_X$ always exists, cf. also \cite[Theorem~2.0.1]{Wlo08}. We deduce:

\begin{corollary}\label{corollary resolutions}
Let $X$ be a normal compact complex variety and let $\sX\to\Def^\lt(X)$ be a miniversal locally trivial deformation.  Then there is a locally trivial log resolution $\bar\pi:\sY\to \sX$ such that $\bar\pi_*T_{\sY/\Def^\lt(X)}=T_{\sX/\Def^\lt(X)}$.
\end{corollary}
Note that by this we mean $\bar\pi:\sY\to \sX$ is a locally trivial resolution which is fiberwise a log resolution, which by the local triviality of $\bar\pi$ is equivalent to the special fiber being a log resolution.  Moreover, it follows that the inclusion $\sD\to\sY$ of the exceptional divisor (and even the map $\sD\to\sX$) is locally trivial. 

\begin{proof}Set $S=\Def^\lt(X)$ and take a log resolution $\pi:Y\to X$ for which $\pi_{*}T_{Y}=T_X$.  From the proposition, we have a formal deformation $\wh\sY\to \wh S$ of $Y$ and a locally trivial formal deformation $\wh\pi:\wh\sY\to\wh\sX$ of $\pi$ over $\wh S$.  We may trivialize $\sX/ S$ on a Stein cover $\sX_i$ of $\sX$ so that we have analytic gluing maps $g_{ij}$ as in the right side of the diagram below
\[\begin{tikzcd}
Y_{i|j}\times \wh S\arrow[d,"\pi\times \id"']\arrow{r}{\phi_{ij}}&Y_{j|i}\times \wh S\arrow{d}{\pi\times\id}&Y_{i|j}\times S\arrow[d,"\pi\times \id"']\arrow[r,dashed,"f_{ij}"]&Y_{j|i}\times S\arrow{d}{\pi\times\id}\\
X_{i|j}\times \wh S\arrow{r}{\wh g_{ij}}&X_{j|i}\times \wh S&X_{i|j}\times S\arrow{r}{g_{ij}}&X_{j|i}\times S
\end{tikzcd}\]
where $\sX_{i|j}:=\sX_i\times_\sX \sX_j$, $X_{i|j}:=\sX_{i|j}\times_\sX S$, $Y_{i|j}:=\pi^{-1}(X_{i|j})$, and all morphisms are $S$-morphisms.  By the proof of the proposition, the resulting cocycle $(X_{i},\wh g_{ij})$ of $\mathbf{U}_X(\wh \sO_S):=\varprojlim \mathbf{U}_{X}(\sO_S/\mathfrak{m}_S^k) $ gives a cocycle for $\mathbf{U}_{Y/X}(\wh \sO_S)$ (analogously defined), as in the left part of the diagram.  As $\pi$ is an isomorphism on a Zariski open set and the $Y_{j|i}\times S$ are separated, the $f_{ij}$ in the diagram on the right are locally uniquely determined, and if they exist they satisfy the cocycle condition.  By \cite[Theorem (1.5)(ii)]{Art68}, the $f_{ij}$ exist locally, and so by the previous remark we obtain a locally trivial resolution $\bar\pi:\sY\to\sX$.  Since the $\phi_{ij}$ in the left diagram are also uniquely determined by $\wh g_{ij}$, the map $\sY\to\sX$ in fact completes to $\wh\sY\to\wh\sX$.
\end{proof}

Recall that the {\it Fujiki class} $\scrC$ consists of all those compact complex varieties which are meromorphically dominated by a compact K\"ahler manifold, see \cite[\S 1]{Fuj78}.  Recall also that a normal complex variety $X$ has {\it rational singularities} if for some (hence any) resolution $\pi:Y\to X$ we have $R^q\pi_*\sO_Y=0$ for $q>0$.

\begin{corollary}\label{cor hodge numbers} Let $\sX\to S$ be a locally trivial family of normal varieties of Fujiki class $\mathscr{C}$ with rational singularities.  Then for all $p$ the function $s\mapsto h^0(X_s,\Omega^{[p]}_{X_s})$ is locally constant.
\end{corollary}
\begin{proof}  Let $\sY\to\sX$ be a locally trivial resolution.  By Kebekus--Schnell \cite[Corollary~1.8]{KS18} we have $h^0(Y_s,\Omega^p_{Y_s})=h^0(X_s,\Omega^{[p]}_{X_s})$, and so the claim follows from the local constancy of Hodge numbers in smooth families of Fujiki class manifolds. Note that as for Kähler manifolds, the constancy of Hodge numbers follows from the degeneration of the Hodge-to-de Rham spectral sequence \cite[Proposition~1.3]{Uen83} via Deligne's classical argument \cite[Th\'eor\`eme~5.5]{Del68}.
\end{proof}
Proposition~\ref{prop morph deforms} in particular applies to crepant bimeromorphic morphisms of symplectic varieties; this greatly simplifies the approach of \cite[Section~4]{BL16} and yields a generalization of \cite[Proposition~4.5]{BL16}:
\begin{corollary}  Let $Y$ be a compact normal complex variety with a nowhere degenerate form $\sigma\in H^0(Y^\reg,\Omega_{Y^\reg}^2)$ and let $\pi:Y \to X$ be a proper bimeromorphic map to a normal variety $X$ with rational singularities.  Then there is a locally trivial deformation $\sY \to \sX$ of $\pi$ over $\Def^\lt(X)$, where $\sX\to\Def^\lt(X)$ is the miniversal locally trivial deformation of $X$.
\end{corollary}
\begin{proof}The symplectic form $\sigma$ induces a symplectic form on $X^\reg$, and we have $T_Y\cong \Omega^{[1]}_Y$ and $T_X\cong \Omega^{[1]}_X$ by reflexivity.  Vector fields therefore lift from $X$ to $Y$ as reflexive $1$-forms do by Kebekus--Schnell \cite[Corollary~1.8]{KS18}.  The existence of the formal deformation of $\pi$ then follows from Proposition \ref{prop morph deforms}, and it is effectivized as in the proof of Corollary \ref{corollary resolutions}.
\end{proof}

\subsection{Deformations along weakly symplectic split foliations}\label{section symplectic foliations}
\newcommand{\reflextimes}{\raisebox{.08em}{\resizebox{.3em}{.5em}{[}}\hspace{-.3em}\otimes\hspace{-.3em}\raisebox{.08em}{\resizebox{.3em}{.5em}{]}}}
\def\rad{\operatorname{rad}}
In this section we show that locally trivial deformations along certain split foliations are unobstructed, see Proposition~\ref{prop EP unobstructed} and Corollary~\ref{corollary DefE germ}.  The following definitions will be useful.
\begin{definition}\label{def good}Let $\sX/S$ be a locally trivial family of normal varieties.  We say a splitting $T_{\sX/S}=\oplus_i E^{(i)}_{\sX/S}\oplus P_{\sX/S}$ is \emph{good} if:
\begin{enumerate}
\item $E_{\sX/S}:=\oplus_i E_{\sX/S}^{(i)}$ is a foliation;
\item For each $i$ we have either:
\begin{enumerate}
\item $E^{(i)}_{\sX/S}$ admits a nondegenerate reflexive 2-form; or
\item $E^{(i)}_{\sX/S}\cong E^{(j)\vee}_{\sX/S}$ for some $j\neq i$.
\end{enumerate}
\end{enumerate}
We will refer to $E_{\sX/S}$ as a \emph{weakly symplectic split foliation}.
\end{definition}
\begin{remark}\label{remark good forms}Observe that the second condition in the definition is equivalent to:
\begin{enumerate}
\item[$(2')$]  For each $i$ there is a reflexive 2-form $\sigma_S^{(i)}\in H^0(\sX,\Omega^{[2]}_{\sX/S})$ such that contraction with $\sigma_S^{(i)}$ yields an isomorphism $E^{(i)}_{\sX/S}\to E^{(j)\vee}_{\sX/S}$ for some $j$.
\end{enumerate}
This is justifies the term ``weakly symplectic split foliation'': $E_{\sX/S}$ is not necessarily a symplectic sheaf itself but every summand in the splitting is either symplectic or embedded in a symplectic subsheaf.
\end{remark}
\begin{definition}Let $\sX/S$ be a locally trivial family of normal varieties.  Given $\tau\in H^0(\sX,\Omega^{[p]}_{\sX/S})$, we define the \emph{radical}
\[\rad(\tau):=\{t\in T_{\sX/S}\mid \iota_t\tau=0\}\subset T_{\sX/S}.\]
\end{definition}  
\begin{remark}\label{remark rad is fol}
Note that if $\tau$ is a reflexive $p$-form on a normal compact variety $X$ of Fujiki class $\mathscr{C}$ with rational singularities, then $E=\rad(\tau)$ is automatically a foliation.  Indeed, for any $u,v\in E$ we have 
\begin{align*}
\iota_{[u,v]}\tau&=\underbrace{L_u \iota_v \tau}_{=0} -\iota_vL_u\tau \\
&=-\iota_v\iota_ud\tau-\iota_vd\iota_u\tau=-\iota_v\iota_ud\tau=0.
\end{align*}
since $d\tau=0$ by Kebekus--Schnell \cite[Corollary~1.8]{KS18}.  In particular, any splitting of $T_X$ into $\sO_X$-modules of trivial determinant is a splitting into foliations.
\end{remark}

The main result of this section is the following:

\begin{proposition}\label{prop EP unobstructed}Let $X$ be a compact complex variety of Fujiki class $\mathscr{C}$ with rational singularities and a good splitting $T_X=\oplus_i E^{(i)}\oplus P$.  Then $F^\lt_{E;\split}$ is unobstructed.
\end{proposition}
\begin{proof}For $(\sX/S)\in F^\lt_{E;\split}(A)$, the deformation module of $F^\lt_{E;\split}$ is canonically a functorial quotient of $\HH^1(\sX,M^{E;\split}_{\sX/S})$ by Corollary~\ref{corollary defo split}.  We will use the $T^1$-lifting criterion of \cite{Ran92,Kaw92,Kaw97a}.  While $F^\lt_{E;\split}$ is not necessarily pro-representable, by Remark~\ref{remark H5} we may use the version in \cite[Theorem 1.8]{Gro97b}.  The claim follows once we know for any $(\sX/S)\in F^\lt_{E;\split}(A)$ and any lift $(\sX'/S')\in F^\lt_{E;\split}(A')$ through a small extension $A'\to A$ that the restriction map
\[
\HH^1(\sX',M^{E;\split}_{\sX'/S'})\to \HH^1(\sX,M^{E;\split}_{\sX/S})
\]
is surjective.

\begin{step}\label{step 1}
For any $A\in \Art_\C$ and for any $(\sX/S)\in F^\lt_{E;\split}(A)$, the splitting $T_{\sX/S}=\oplus_i E^{(i)}_{\sX/S}\oplus P_{\sX/S}$ is good.
\end{step}
It suffices to assume the splitting is good for $(\sX/S)\in F^\lt_{E;\split}(A)$ and to show that the splitting of any lift $(\sX'/S')\in F^\lt_{E;\split}(A')$ through a small extension $A'\to A$ is good.  Note that in the notation of Remark \ref{remark good forms}, if $\sigma^{(i)}_S$ lifts to $\sigma^{(i)}_{S'}\in H^0(\sX',\Omega^{[2]}_{\sX'/S'})$, then the induced map $E^{(i)}_{\sX'/S'}\to E^{(j)\vee}_{\sX'/S'}$ is also isomorphism as both are flat.  

We therefore reduce to showing the $\sigma^{(i)}_{S}$ lift.  For this we could appeal directly to the degeneration of reflexive Hodge-to-de Rham in low degrees \cite[Lemma~2.4]{BL18}, but we also include a more direct argument using Corollary~\ref{corollary resolutions}.  Let $\pi':\sY'\to\sX'$ be a simultaneous locally trivial resolution with special fiber $\bar\pi:Y\to X$.  By Kebekus--Schnell \cite[Corollary~1.8]{KS18} we have $\bar\pi_{*}\Omega^2_Y=\Omega_X^{[2]}$ via the natural map.  By local triviality we then have $\pi'_*\Omega^2_{\sY'/S'}=\Omega_{\sX'/S'}^{[2]}$ via the natural map, as both are flat and the natural map is an isomorphism on the special fiber.  By Deligne \cite[Théorème~5.5]{Del68} (see also e.g. \cite[Lemma~2.4]{BL18} for the necessary changes in the analytic category), $H^0(\sX',\Omega^{[2]}_{\sX'/S'})=H^0(\sY',\Omega^2_{\sY'/S'})$ is free and compatible with base-change, and it follows that the restriction map
$H^0(\sX',\Omega^{[2]}_{\sX'/S'})\to H^0(\sX,\Omega^{[2]}_{\sX/S})$ is surjective.

\begin{step}\label{step 2}  
For any $A\in \Art_\C$, any $(\sX/S)\in F^\lt_{E,P}(A)$, and any lift $(\sX'/S')\in F^\lt_{E,P}(A')$ through a small extension $A'\to A$ the restriction map
\[
H^1(\sX',E_{\sX'/S'})\to H^1(\sX,E_{\sX/S})
\]
is surjective.
\end{step}

From \cite[Lemma~2.4]{BL16} the restriction map $H^1(\sX',\Omega^{[1]}_{\sX'/S'})\to H^1(\sX,\Omega^{[1]}_{\sX/S})$ is surjective.  As the splitting of $\sX'/S'$ is good by the previous step, for each $i$ we have that $E^{(i)}_{\sX'/S'}$ is a direct summand of $\Omega^{[1]}_{\sX'/S'}\cong T_{\sX'/S'}^\vee$ which is compatible with restriction, and therefore the restriction map $H^1(\sX',E^{(i)}_{\sX'/S'})\to H^1(\sX,E^{(i)}_{\sX/S})$ is surjective.  The claim then follows.

\begin{step}
\label{step ses}For any $(\sX/S)\in F^\lt_{E;\split}(A)$ the natural sequence
\begin{align*}
0&\to \bigoplus_i\Hom_{\sO_\sX}(E^{(i)}_{\sX/S},E_{\sX/S}/E^{(i)}_{\sX/S})\oplus \Hom_{\sO_\sX}(P_{\sX/S},T_{\sX/S}/P_{\sX/S})\\
&\to \HH^1(\sX,M^{E;\mathrm{split}}_{\sX/S})\to H^1(\sX,E_{\sX/S})\to 0
\end{align*}
is exact.
\end{step}  For notational simplicity set
\[\mathcal{H}_{\sX/S}:=\bigoplus_i\sHom_{\sO_\sX}(E^{(i)}_{\sX/S},E_{\sX/S}/E^{(i)}_{\sX/S})\oplus \sHom_{\sO_\sX}(P_{\sX/S},T_{\sX/S}/P_{\sX/S})\]
and $H_{\sX/S}$ for the global sections.  Using the long exact sequence associated to the triangle 
\[M^{E;\mathrm{split}}_{\sX/S}\to E_{\sX/S}\xrightarrow{\oplus_i\mathrm{ad}^{E;E^{(i)}}\oplus\mathrm{ad}^{X;P}} \mathcal{H}_{\sX/S}\to M^{E;\mathrm{split}}_{\sX/S}[1]\]
its sufficient to show that the induced maps 
\begin{align*}
\mathrm{ad}^{E;E^{(i)}}:H^q(\sX,E_{\sX/S})&\to H^q(\sX,\sHom_{\sO_\sX}(E^{(i)}_{\sX/S},E_{\sX/S}/E^{(i)}_{\sX/S}))\\
\mathrm{ad}^{X;P}:H^q(\sX,E_{\sX/S})&\to H^q(\sX,\sHom_{\sO_\sX}(P_{\sX/S},T_{\sX/S}/P_{\sX/S}))
\end{align*} vanish for $q=0,1$.  Temporarily set $E^{(0)}_{\sX/S}:=P_{\sX/S}$ and let $\pi_i:T_{\sX/S}\to E^{(i)}_{\sX/S}$ be the projection.  Its then sufficient to show the induced map
\[\pi_j\circ\mathrm{ad}:H^q(\sX,T_{\sX/S})\to H^q(\sX,\sHom_{\sO_\sX}(E^{(i)}_{\sX/S},E^{(j)}_{\sX/S}))\]
vanishes for $q=0,1$ and all $i\neq j$ with $j\neq 0$. Together with the degeneration of reflexive Hodge-to-de Rham $H^q(\sX,\Omega_{\sX/S}^{[p]})\Rightarrow \mathbb{H}^{p+q}(\sX,\Omega^{[\bullet]}_{\sX/S})$ in the range $p+q\leq 2$ \cite[Lemma~2.4]{BL16}, it is enough to show the following:
\begin{claim} For each $i\neq j$ with $j\neq 0$ we have a commutative diagram
\begin{equation}\label{eq d sq}\begin{tikzcd}
T_{\sX/S}\arrow{r}{\pi_j\circ \mathrm{ad}}\arrow{d}& \sHom_{\sO_\sX}(E^{(i)}_{\sX/S},E^{(j)}_{\sX/S})\\
\Omega_{\sX/S}^{[1]}\arrow{r}{d}&\Omega^{[2]}_{\sX/S}.\arrow{u}
\end{tikzcd}\end{equation}

\end{claim}
\begin{proof}Fix $i,j$.  Using Step \ref{step 1}, there is a reflexive 2-form $\sigma_S$ whose radical contains $E^{(i)}_{\sX/S}$ and which by contraction gives an isomorphism $E^{(j)}_{\sX/S}\to E^{(j')\vee}_{\sX/S}$.  We then take the left vertical map of \eqref{eq d sq} to be the contraction map $t\mapsto \iota_t\sigma_S$ and the right vertical map to associate to a form $\alpha$ the map $f\in \sHom_{\sO_\sX}(E^{(i)}_{\sX/S},E^{(j)}_{\sX/S})$ such that $\sigma_S(f(u),v) =-\alpha(u,v)$ for $u\in E^{(i)}_{\sX/S}$ and $v\in E^{(j')}_{\sX/S}$.  The latter is none other than the projection (up to a sign) of $\Omega^{[2]}_{\sX/S}$ onto the $E^{(i)\vee}_{\sX/S}\reflextimes E^{(j')\vee}_{\sX/S}$ factor in its natural splitting using the identification $E^{(j)}_{\sX/S}\cong E^{(j')\vee}_{\sX/S}$.

For the commutativity of \eqref{eq d sq} we need to show for sections $t$ of $T_{\sX/S}$, $u$ of $E^{(i)}_{\sX/S}$, and $v$ of $E^{(j')}_{\sX/S}$ that
\[(d\iota_t\sigma_S)(u,v)=-\sigma_S([t,u],v).\]
On the one hand, since $\sigma_S$ is closed (again by the low degree degeneration of reflexive Hodge-to-de Rham), we have
\[L_t\sigma_S=d\iota_t\sigma_S\]
by the Cartan formula.  On the other hand, since $\sigma_S$ vanishes on $E^{(i)}_{\sX/S}$ we have
\begin{align*}
(L_t\sigma_S)(u,v)&=t.\underbrace{\sigma_S(u,v)}_{=0}-\sigma_S(L_tu,v)-\underbrace{\sigma_S(u,L_tv)}_{=0}\\
&=-\sigma_S([t,u],v).
\end{align*}
\end{proof}

\begin{step}
Final step of the proof.
\end{step}
Now for $\sX'/S'$ lifting $\sX/S$, we have a natural diagram
 \[\begin{tikzcd}
0\ar[r]& H_{\sX'/S'}\ar[r]\ar[d]& \HH^1(\sX,M^{E;\split}_{\sX'/S'})\ar[r]\ar[d] &H^1(\sX',E_{\sX'/S'})\ar[r]\ar[d]& 0\\
0\ar[r]& H_{\sX/S}\ar[r]& \HH^1(\sX,M^{E;\split}_{\sX/S})\ar[r] &H^1(\sX,E_{\sX/S})\ar[r]& 0
\end{tikzcd}\]
using the notation of the previous step, where the vertical maps are restrictions.  The right vertical map is surjective by Step \ref{step 2}, while the left vertical map is surjective since each $\Hom_{\sO_{\sX'}}(E^{(i)}_{\sX'/S'},E^{(j)}_{\sX'/S'})$ (again using the notation $E^{(0)}_{\sX'/S'}=P_{\sX'/S'}$) is a summand of $H^0(\sX',\Omega_{\sX'/S'}^{[2]})$ that is compatible with the restriction map $H^0(\sX',\Omega_{\sX'/S'}^{[2]})\to H^0(\sX,\Omega_{\sX/S}^{[2]})$, which is surjective as in Step \ref{step 1}.  The rows are exact by Step \ref{step ses}, and it follows that the middle vertical map is surjective, thus completing the proof.
\end{proof}

\begin{remark}Without additional assumptions, Proposition \ref{prop EP unobstructed} does not obviously generalize to arbitrary $K$-trivial splittings, and in particular cannot be adapted to prove a locally trivial Bogomolov--Tian--Todorov theorem.  For such a splitting, each factor is identified with a direct summand of some $\Omega^{[p-1]}_{X}$ via contraction with a $p$-form, and Step \ref{step ses} of the proof still carries through.  Indeed, the map $d:\Omega^{[p-1]}_{\sX/S}\to\Omega^{[p]}_{\sX/S}$ still induces the zero map on zeroth and first cohomology using the Leray spectral sequence and the result of Kebekus--Schnell.  On the other hand, Step \ref{step 2} may fail in general as reflexive Hodge-to-de Rham does not necessarily degenerate in higher degrees (although see \cite[Theorem~3.4]{Dan91} and \cite[(1.12)~Theorem]{Ste77} for some special cases over a point). 
\end{remark}
\begin{corollary}\label{corollary DefE germ} Assume the hypotheses of the previous proposition.  Then $F^\lt_{E;\split}\to F^\lt_E$ is formally smooth and the functor $F^\lt_E$ is unobstructed.  In particular, there exists a locally trivial deformation $\scrX\to(H^1(X,E),0)$ of $X$ whose Kodaira--Spencer map is the natural map $H^1(X,E)\to H^1(X,T_X)$.
\end{corollary}
\begin{proof}As $F^\lt_{E;\split}\to F^\lt_E$ is surjective on tangent spaces by Step \ref{step ses} of the proof of Proposition~\ref{prop EP unobstructed} and $F^\lt_{E;\split}$ is unobstructed, it follows easily by induction on small extensions that $F^\lt_{E;\split}\to F^\lt_E$ is formally smooth.  The unobstructedness of $F^\lt_{E;\split}$ then immediately implies that $F^\lt_E$ is unobstructed.  By Corollary \ref{cor def fol} we have a map on the level of formal spaces $\widehat{(H^1(X,E),0)}\to\wh\Def^\lt(X)$ with the required derivative, and by Artin approximation \cite[Theorem (1.2)]{Art68} there is a map on the level of analytic germs with the required derivative.
\end{proof}

\section{K-trivial varieties and strong approximations}\label{section ktrivial}

Let us fix terminology. For a normal variety $X$, we denote by $\omega_X$ the double dual of $\det \Omega_X^1$ and by $\omega_X^{[m]}:=(\omega_X^{\otimes m})^{\vee\vee}$ its reflexive powers.

\begin{definition}\label{definition ktrivial}
A \emph{numerically $K$-trivial} variety is a normal complex variety $X$ with rational singularities such that $\omega_X^{[m]}$ is a line bundle for some $m>0$ and $c_1(\omega_X)=0$ as an element of $H^{2}(X,\Q)$. If $\omega_X$ satisfies $\omega_X^{[m]}\cong\sO_X$ for some $m>0$ we say that $X$ is \emph{$K$-torsion}.  We say $X$ is \emph{$K$-trivial} if $\omega_X^{\vee\vee}\cong\sO_X$.
\end{definition}

\begin{remark}\label{remark flat bundle}
Let $X$ be a compact Kähler space with log terminal singularities. By \cite[Corollary~1.18]{JHM2}, numerical $K$-triviality is equivalent to $X$ being $K$-torsion. Actually, the proof does not rely on $X$ being Kähler but rather on $X$ admitting a Kähler resolution; in particular, the result holds when $X$ is merely in Fujiki class $\mathscr C$. Note that by normality, $\omega_X^{[m]}\cong\sO_X$ if and only if $\omega_{X^\reg}^m\cong\sO_{X^\reg}$.  Moreover, by the result of Kebekus--Schnell (\cite[Corollary~1.8]{KS18}), if $X$ is normal with rational singularities and $\omega_X^{\vee\vee}\cong\sO_X$, then it has canonical singularities (and is in particular $K$-trivial in the above sense).
\end{remark}
It is easy to construct examples of numerically $K$-trivial varieties in the sense of Definition \ref{definition ktrivial}.  For example, any anti-canonical divisor with rational singularities in a Gorenstein variety is $K$-trivial.  K\"ahler varieties with symplectic singularities in the sense of Beauville provide some other examples (in particular, see the primitive symplectic varieties of \cite{BL18}).

\subsection{Proof of Theorem \ref{theorem peternell}}
We are now ready to prove Theorem \ref{theorem peternell}.  Given the results of Section~\ref{section symplectic foliations}, the main step is to show the following:
\begin{lemma}\label{lemma symp splitting}  Let $X$ be a numerically $K$-trivial compact Kähler variety with log terminal singularities.  Then there exists a good splitting $T_X=\oplus_i E^{(i)}\oplus P$ in the sense of Definition \ref{def good} where $P$ is the common radical of all reflexive 2-forms.
\end{lemma}

\begin{proof}  
We fix a Kähler class on $X$; in the following, we always mean stability with respect to this class.  By \cite[Theorem A(ii)]{GSS} there is a splitting $T_X=\bigoplus_{i\in I} E_i$ into stable factors with slope zero.  Note that for any reflexive 2-form $\sigma$ we have a contraction map $T_X\to \Omega^{[1]}_X\cong T_X^\vee$, and by stability of the factors it follows that for each $i$ either $E_i\subset\rad(\sigma)$ or contraction with $\sigma$ induces an isomorphism $E_i\to E_j^\vee$ for some $j$.  In particular, the common radical is $P=\bigoplus_{i\in I_P}E_i$ for a subset $I_P\subset I$.  Finally, for each $i$ we have that $E_i|_{X^\reg}$ is parallel with respect to the Ricci-flat metric in the chosen K\"ahler class, so any sum of the $E_i$ is a foliation.  We therefore conclude that
\[T_X=\bigoplus_{i\in I\setminus I_P}E_i\oplus P\]
is a good splitting.
\end{proof}

By the lemma, the hypotheses of Proposition \ref{prop EP unobstructed} are satisfied and Corollary \ref{corollary DefE germ} applies.  In view of Graf--Schwald \cite[Theorem 3.1]{GS20}, the proof is completed by the following:
\begin{lemma}\label{lemma approx crit}
For $X$ as in the previous lemma, suppose there is a splitting 
$T_X=E\oplus P$
such that every reflexive 2-form contains $P$ in its radical.  Then for any K\"ahler class $\omega$ on $X$, the contraction map
\[\iota_{\omega}:  H^1(X,E) \to H^2(X,\sO_X)\]
is surjective.
\end{lemma}
We refer to \textsection~\ref{section forms and currents} for Kähler forms on singular spaces. To any such class $\omega$ on $X$, one may associate an element $\kappa(\omega) \in H^1(X,\Omega_X^1)$, cf. \eqref{diff map}. We then denote by $\iota_\omega$ the composition of the cup product $H^1(X,T_X)\overset{\kappa(\omega) \cup \--}{\longrightarrow} H^2(X,T_X\otimes \Omega_X^1)$ and the contraction $H^2(X,T_X\otimes \Omega_X^1)\to H^2(X,\sO_X)$. Similarly, if $t\in H^1(X,T_X)$, we have a contraction map $\iota_t: H^0(X,\Omega_X^{[2]})\overset{t \cup \--}{\longrightarrow} H^1(X,\Omega_X^{[2]}\otimes T_X)\to H^1(X,\Omega_X^{[1]})$.

\begin{proof}Set $n:=\dim X$; one may assume $n>1$. The contraction map \[\iota_{\omega}:  H^1(X,T_X) \to H^2(X,\sO_X)\] is already surjective by \cite[Theorem~4.1]{GS20}, and the proof of the lemma follows by a similar computation.  We have a perfect pairing
\[H^0(X,\Omega_X^{[2]})\otimes H^2(X,\sO_X)\to \C:(\alpha,\beta)\mapsto \int_X\omega^{n-2}\wedge \alpha\wedge\beta \]
and so $\iota_{\omega}$ is identified with the surjective map $f:H^1(X,T_X)\to H^0(X,\Omega_X^{[2]})^\vee$ given by 
\[f(t):\alpha\mapsto \int_X \omega^{n-2}\wedge \alpha \wedge \iota_t(\omega)=-\frac{1}{n-1}\int_X \omega^{n-1}\wedge \iota_t(\alpha)\]
 using that $\omega^{n-1}\wedge \alpha=0$ and \[\iota_t(\omega^{n-1}\wedge \alpha)=(n-1)\, \omega^{n-2}\wedge \iota_t(\omega)\wedge\alpha+\omega^{n-1}\wedge \iota_t(\alpha).\]Now any $\alpha\in H^0(X,\Omega_X^{[2]})$ vanishes on $P$, so $f$ kills $H^1(X,P)$, whence the lemma.
\end{proof}
The proof of Theorem~\ref{theorem peternell} is now complete.
\qed

\subsection{Quasi-\'etale covers in families and applications}
In this section we deduce some first consequences of Theorem \ref{theorem peternell}. We refer to the introduction for the definitions of quasi-\'etale maps and quasi-\'etale covers.  We first prove two lemmas asserting that quasi-\'etale covers can be spread out in locally trivial families, and that the resulting families are also locally trivial.  Because it is all we will need, for simplicity we only consider families over the disk $\Delta$.

Recall that a quasi-\'etale cover restricts to an \'etale cover of $X^\reg$, and conversely any \'etale cover of $X^\reg$ can be extended to a quasi-\'etale cover of $X$ \cite[Thm.~3.4]{DG94}.

\begin{lemma}\label{lem:cover}
Let $\pi:\sX\to\Delta$ be a locally trivial family of normal varieties.
\begin{enumerate}
\item  Let $f:\sY \to \sX$ be a quasi-étale cover. Then for any $t\in \Delta$, $f$ induces a quasi-étale cover $f_t: Y_t\to X_t$. 

\item 
Conversely, let $t\in \Delta$ and let $f_t: Y_t \to X_t$ be a quasi-étale cover. Then there exists a quasi-étale cover $f:\sY \to \sX$ such that $f|_{Y_t}=f_t$. 

\end{enumerate}
\end{lemma}

\begin{proof}
The first part is straightforward; by Nagata's purity of branch locus, $f$ is étale over $\sX^\reg$ hence $f_t$ is étale over $\sX^\reg \cap X_t = X_t^\reg$. As for the second assertion, note that the canonical morphism $\pi_1(X_t^\reg) \to \pi_1(\sX^\reg)$ induced by the inclusion $X_t^\reg \hookrightarrow \sX^\reg$ is actually isomorphic since $\pi$ is topologically trivial. In particular, any étale cover $Y_t^\reg \to X_t^\reg$ is the restriction of an étale cover $\sY^\reg \to \sX^\reg$ over $\Delta$.
\end{proof}

\begin{lemma}
\label{lem:cover2}
Let $\pi:\sX\to\Delta$ be a locally trivial family of normal varieties and $f:\sY \to \sX$ a quasi-étale cover. Then $f\circ \pi: \sY \to \Delta$ is locally trivial.
\end{lemma}

\begin{proof}
One can cover $\sX$ by (small) open subsets $U$ where $\pi$ is trivial; i.e. $U\isom U_0\times \Delta$ over $\Delta$ where $U_0=X_0\cap U$. We claim that up to shrinking $U$, the space $\sY|_{f^{-1}(U)}$ is isomorphic over $\Delta$ to $V_0\times \Delta$ where $V_0\to U_0$ is a suitable quasi-étale cover. 

 Since there is a $1-1$ correspondence between quasi-étale covers of $U$ and étale covers of $U^{\rm reg}$ (or, equivalently, finite index subgroups of $\pi_1(U^{\rm reg})$), the result follows from the isomorphism $\pi_1(U^{\rm reg}) \isom \pi_1(U_0^{\rm reg})$. 
 
 Indeed, set $V:=f^{-1}(U)$ and $V^\circ:=f^{-1}(U^{\rm reg})$. The cover $f$ induces an étale cover $f|_{V^\circ}:V^\circ \to U_0^{\rm reg}\times \Delta$. By the observation above, there is an isomorphism of étale covers $V^\circ \isom W_0\times \Delta$ for some étale cover $h_0:W_0\to U_0^{\rm reg}$. We can uniquely extend $h_0$ into a quasi-étale cover $h:W\to U_0$ of normal spaces and we will automatically obtain an isomorphism $V\isom W\times \Delta$ of covers over $U_0\times \Delta$. In particular, $W$ is isomorphic to $V\cap \mathcal Y_0$, and $f\circ \pi$ is indeed locally trivial as claimed.
\end{proof}

\begin{corollary}
\label{cor mqe}
Let $X$ be a numerically K-trivial variety with log terminal singularities. Then $X$ admits a maximally quasi-étale cover $\widetilde X\to X$; i.e. the natural morphism $\hat{\pi}_1(\widetilde X^{\rm reg}) \to \hat \pi_1(\widetilde X)$ is an isomorphism. 
\end{corollary}

\begin{proof}
By Theorem~\ref{theorem peternell}, there is a locally trivial deformation $\sX\to \Delta$ of $X$ such that for some $t\in \Delta$, the fiber $X_t$ is projective. By \cite[Theorem~1.5]{GKP16etale},  there exists a maximally quasi-étale cover $\wt X_t \to X_t$. By Lemma~\ref{lem:cover}, we can find a quasi-étale cover $\wt \sX\to \sX$ inducing $\wt X_t \to X_t$ and a quasi-étale cover $\wt X \to X$. By Lemma~\ref{lem:cover2}, the cover $\wt\sX \to \sX$ is a locally trivial family, so by \cite[Proposition~6.1]{AV19} there is a homeomorphism $\wt X\overset{\sim}{\longrightarrow} \wt X_t$ inducing a diffeomorphism $\wt X^{\rm reg} \overset{\sim}{\longrightarrow} \wt X_t^{\rm reg}$ by restriction. The corollary follows.
\end{proof}

Recall that for $X$ a compact K\"ahler variety with rational singularities, the augmented irregularity $\wt q(X)$ \cite[Definition 2.1]{CGGN}  is 
\[
\wt q(X)=\sup\left\{h^1(\qeX,\sO_{\qeX})=h^0(\qeX,\Omega^{[1]}_\qeX)\mid \qeX\to X\mbox{ quasi-\'etale cover}\right\}\in \N \cup \{\infty\}.
\]
If additionally $X$ has klt singularities and trivial first Chern class, then $\wt q(X) \le \dim X$; in particular, it is finite \cite[Corollary~4.2(i)]{CGGN}.

\begin{corollary}\label{cor irred constant}Let $\sX\to \Delta$ be a locally trivial deformation of a $K$-trivial K\"ahler variety $X$ with canonical singularities.  Then the augmented irregularity $\wt q(X_t)$ is constant, and one fiber is IHS (resp. ICY) if and only if every fiber is IHS (resp. ICY).
\end{corollary}
\def\Sp{\mathrm{Sp}}
\begin{proof}The first statement is immediate from the lemmas and Corollary \ref{cor hodge numbers}.  

For the second part, assume that one fiber, say $X_0$, is ICY or IHS. This implies that $\wt q(X_0)=0$, and $\wt q(X_t)=0$ for all $t$. We fix $t\in \Delta$. By \cite[Theorem C \&~Proposition~6.9]{CGGN} and Corollary~\ref{cor mqe}, there exists a quasi-étale cover $\wt X_t\to X_t$ such that the tangent sheaf of $\wt X_t$ has a decomposition $T_{\wt X_t}=(\oplus_{i\in I} C_i) \oplus (\oplus_{j\in J} S_j)$ where $C_i$ (resp. $S_j$) is of ICY type (resp. IHS type). Recall that this means that each $C_i$ (resp. $S_j$) has full holonomy $\SU(\rk\, C_i)$ (resp. $\Sp(\rk\, S_j)$) with respect to a singular Ricci-flat metric. By Lemma~\ref{lem:cover}, one can find a quasi-étale $\wt \sX\to \sX$ extending $\wt X_t\to X_t$, and by Lemma~\ref{lem:cover2} the family $\wt X \to \Delta$ is a locally trivial deformation. 

If one can show that $T_{\wt X_t}$ has a single summand, then $\wt X_t$ will be either ICY or IHS by \cite[Corollary~E]{CGGN} and the type of $\wt X_t$ will then be determined by the Hodge numbers $t\mapsto h^0(\wt X_t,\Omega_{X_t}^{[p]})$ which are independent of $t$, cf. Corollary \ref{cor hodge numbers}. By the same argument, $X_t$ itself will be ICY or IHS, with type determined by that of $X_0$.

Since each of the factors in the decomposition of $T_{\wt X_t}$ accounts for at least one holomorphic form (in maximal rank), there can be only one such factor if $X_0$ is ICY. 

It now remains to show the corollary in the case where $X_0$ is IHS. Write $\dim X_0=2n$. First, one observes that $|J|=h^0(\wt X_t, \Omega_{\wt X_t}^{[2]})=1$. Let $\sigma \in H^0(\wt X_t, \Omega_{\wt X_t}^{[2]})$ a nonzero element and let $m:=\frac 12 \mathrm{rank}(S_1)$. We have $m\le n$, and we want to show that this inequality is actually an equality. Clearly, one has  $\sigma^{m+1}= 0$.  Remember that the cup product map $\Sym^{k}H^2(\wt X_0,\Q)\to H^{2k}(\wt X_0,\Q)$ is injective for any $k\le n$, cf. \cite[Proposition~5.15]{BL18}, so the same statement is true for $\wt X_t$ as locally trivial families (over $\Delta$) are topologically trivial. This implies that $m+1>n$, and given that $n \ge m$, we have $n=m$ as desired.
\end{proof}

\section{Reminder on the Douady space}\label{section douady}

For a proper morphism $f:\sX\to S$ of complex spaces we denote by $\scrD(\sX/S) \to S$ the relative Douady space constructed by Pourcin \cite{Pou69}, who generalized the work of Douady \cite{Dou66} to the relative situation. Let us denote by $\ol\scrD_q(\sX/S)\subset \scrD(\sX/S)$ the union of all irreducible components whose general element parametrizes a pure dimensional reduced subspace of dimension $q$. Let $\scrB_q(\sX/S)$ be the relative Barlet space of dimension $q$ cycles in the fibers of $f$, see \cite[Théorème 5]{Bar75}. By \cite[Théorème 8]{Bar75}, there is a morphism of complex spaces
\begin{equation}\label{eq barlet morphism}
\vrho:(\ol\scrD_q(\sX/S))_\red \to \scrB_q(\sX/S)
\end{equation}
where $(\ol\scrD_q(\sX/S))_\red$ denotes the reduction of $\ol\scrD_q(\sX/S)$.

\begin{remark}\label{remark equidimensionality douady}
As the following example shows, pure dimensionality is not an open property. The Hilbert scheme $H=\Hilb^{3n+1}(\P^3)$ of closed subschemes of $\P^3$ with Hilbert polynomial $p(n)=3n+1$ is a union $H=H'\cup H''$ where $H'$ and $H''$ are smooth and irreducible (of dimension $12$, $15$ respectively) and intersect transversally with $\dim H'\cap H'' = 11$, see the main theorem of \cite{PS85}. Generically, $H'$ parametrizes twisted cubics and $H''$ parametrizes plane cubics in $\P^3$ union an additional point. Both may degenerate to a singular plane cubic with an embedded point at a singularity. In particular, elements of $H'\cap H''$ are pure dimensional (with embedded points) while the general element of $H''$ is not.
\end{remark}

Fujiki has obtained important properness results about the Douady and the Barlet space. These results are fairly comprehensive for Kähler morphisms. We need however a slight generalization to weakly Kähler morphisms. Following \cite{Bin83i} we call a morphism $\sX \to S$ \emph{weakly Kähler} if for any $s_0\in S$, there exists a neighborhood $S^\circ\subset S$ of $S$ and a smooth $(1,1)$-form $\theta$ on $\sX^\circ:=\sX \times_S S^\circ$ such that for all  $s\in S^\circ$, the restriction $\theta|_{X^\circ_s}$ is a Kähler form. Here, $X^\circ_s$ is the fiber at $s$ of $\sX^\circ \to S^\circ$. See \cite[Example~3.9]{Bin83ii} for an example attributed to Deligne showing that a deformation of a Kähler manifold need not be a Kähler morphism.

The following result is proven with exactly the same methods as in Fujiki's article \cite{Fuj78}. We include a sketch of a proof as we could not find it anywhere in the literature. The argument is similar to \cite[Proposition 2.6]{GLR13}.

\begin{proposition}\label{proposition properness douady}
Let $\sX, S$ be reduced and irreducible complex spaces, let $f:\sX\to S$ be a proper locally trivial morphism which is weakly Kähler, let $S^\circ\subset S$ be any simply connected relatively compact subspace, and denote $\sX^\circ := \sX \times_S S^\circ$. Then every irreducible component $D \subset \ol\scrD_q(\sX^\circ/S^\circ)$ respectively $B\subset \scrB_q(\sX^\circ/S^\circ)$ is proper over $S^\circ$.
\end{proposition}
\begin{proof}
By \cite[Proposition~3.4]{Fuj78}, it suffices to show the statement for the Barlet space as the morphism $\vrho:D \to \scrB_q(\sX^\circ/S^\circ)$ will then be proper as well. 
Let $B \subset \scrB_q(\sX^\circ/S^\circ)$ be an irreducible component and denote by $\{F_b\}_{b\in B}$ the universal family of cycles parametrized by $B$. Let $\theta$ be a weak Kähler metric for $f$. The key point is to show that $\lambda(b):=\int_{F_b} \theta^q$ is bounded as $b\in B$ varies, cf. the proof of \cite[Proposition~4.1]{Fuj78}. 

By \cite[Proposition~6.1]{AV19}, the sheaf $(R^{2q}f_*\Q_{\sX^\circ})^\vee$ is a local system, and the argument of \cite[Lemma~19.1.3]{Fulton98} shows that $b\mapsto [F_b]$ defines a section of the pulled back local system along the projection $\pi:B\to S^\circ$. As $S^\circ$ is simply connected, the local system is constant and thus the $[F_b]$ determine a global section $L$ of $(R^{2q}f_*\Q_{\sX^\circ})^\vee$. In particular, the continuous function $\lambda':S^\circ\to \R$ defined by $\lambda'(s):=L_s([\theta_s]^q)$ satisfies $\pi^*\lambda' = \lambda$. As $\lambda'$ can be extended to the closure $\ol S^\circ$ in $S$, we see that $\lambda$ is bounded.
\end{proof}

The following is an immediate consequence.

\begin{corollary}\label{corollary properness douady}
Let $S$ be a reduced and irreducible complex space, let $f:\sX\to S$ be a locally trivial deformation of a compact Kähler variety $X$ with rational singularities. With the notation of Proposition~\ref{proposition properness douady}, every irreducible component $D \subset \ol\scrD_q(\sX^\circ/S^\circ)$ respectively $B\subset \scrB_q(\sX^\circ/S^\circ)$ is proper over $S^\circ$.
\end{corollary}
\begin{proof}
By \cite[Theorem~6.3]{Bin83i}, $f$ is weakly K\"ahler.
\end{proof}

As a consequence of properness, we obtain the following result.

\setcounter{stepcounter}{0}
\begin{lemma}\label{lemma decomposition zariski open}
Let $X$ be a normal compact Kähler variety with rational singularities.  Let $f:\sX \to S$ be a locally trivial deformation of $X$ over an irreducible and reduced base $S$. Suppose that for some nonempty Euclidean open $V\subset S$ we have a product decomposition
\[
\sX\times_S V \isom \sY_V \times_V \sZ_V
\]
for locally trivial $\sY_V,\sZ_V/V$ and suppose that $V$ is contained in a relatively compact simply connected subspace $S^\circ\subset S$.  Then up to replacing $S$ by $S^\circ$, there is a finite morphism $S'\to S$, a Euclidean open $V'\subset S'$ mapping generically isomorphically onto $V$, and a Zariski open $U'\subset S'$ containing $V'$ such that the base change $\sX\times_S U' \to U'$ has a product decomposition
\begin{equation}\label{eq decomposition on zariski open}
\sX\times_S {U'} \isom \sY\times_{U'}  \sZ
\end{equation}
for locally trivial $\sY,\sZ/U'$ that specializes to the pullback of the one of $\sX\times_S V$ over $V'$.
\end{lemma}

\begin{proof}
We may also assume $S$ is normal by passing to the normalization.  We will proceed in three steps.

\begin{step}\label{step douady smooth}
Let $v\in V$ and $z\in Z_v$ be smooth points. Then the relative Douady space $\scrD(\sX/S)$ is smooth over $S$ at $Y(z):=Y_v \times\{z\}$ and $\dim_{[Y(z)]} \scrD(\sX/S) = \dim Z_v + \dim S$.
\end{step}

The normal bundles form a short exact sequence $0\to N_{Y(z)/\sX_v}\to  N_{Y(z)/\sX} \to N_{\sX_v/\sX}\vert_{Y(z)}\to 0$ where the outer terms are trivial bundles of rank $\dim Z_v$ respectively $\dim S$. We deduce
\[
\begin{aligned}
\dim T_{\scrD(\sX/S),[Y(z)]} &= \dim H^0(Y(z),N_{Y(z)/\sX}) \\
&\leq \dim Z_v + \dim S\\
&= \dim \sZ_V \\
& \leq \dim_{[Y(z)]} \scrD(\sX/S).
\end{aligned}
\]
Note that the last inequality follows from the fact that the second projection $\sX\times_S V\to\sZ_V$ can be thought of as the family of fibers $Y(z)$ and the resulting classifying map $\sZ_V\to \scrD(\sX/S)$ is clearly injective.  We infer that equality holds above and $\scrD(\sX/S)$ is smooth at $[Y(z)]$. 
\vskip1em
Before the next step, we set some notation.  By the previous step, there is a unique irreducible component $D \subset \scrD(\sX/S)$ passing through $[Y(z)]\in \scrD(\sX/S)$ and we let $D^\nu \to D$ be its normalization. We denote by $\sF \subset  D^\nu \times_S \sX$ the (pullback of the) universal family and consider the $S$-morphism $e:\sF \to \sX$ induced by projection. Replacing $S$ by $S^\circ$, we may assume that $D^\nu \to S$ is proper by Corollary~\ref{corollary properness douady} and therefore has a Stein factorization $D^\nu\to S'\to S$. Moreover, $S'$ is irreducible and $S'\to S$ is surjective, as $D$ respectively $D \to S$ are. 
\begin{step}\label{step stein factorization}
There is a nonempty Zariski open $U'\subset S'$ and a Euclidean open $V'\subset U' \cap (S'\times_S V)$ such that the induced morphisms $e\times \id:\sF\times_{S'} U' \to \sX\times_S U'$ and $V' \to V$ are isomorphisms. 
\end{step}

Clearly, the classifying map of $\sZ_V$ factors as $\sZ_V \to D \subset \scrD(\sX/S)$. By construction, $\sZ_V \to D$ is injective and as the induced map $\sZ_V\to D\times_S V$ is proper (since $\sZ_V\to V$ is), $\sZ_V$ maps bijectively onto an irreducible component of $D\times_S V$. As $D^\nu \times_S V = (D\times_S V)^\nu$, we have a factorization $\sZ_V\to D^\nu \to D$ such that $\sZ_V$ is isomorphic to a connected component of $D^\nu \times_S V$.  The map $\sZ_V\to V$ factors through an isomorphism $V'\to V$ for a connected component $V'\subset S'\times_S V$. Indeed, $V'\to V$ is finite and bimeromorphic as $\sZ_V\to V$ has connected fibers and $V$ is normal by assumption. Over $V' \subset S'$ we have a diagram

\begin{equation}\label{eq diag restricted douady}
\xymatrix{
\sX \times_{S} V' \ar[r] \ar@/^1.5pc/[rr]^\alpha \ar[d] & \sF \times_{S'} V' \ar[d]\ar[r]  & (\sX \times_{S} S') \times_{S'} V'\ar[dld]\\
\sZ_V \ar[dr] \ar[r]  & D^\nu \times_{S'} V' \ar[d] & \\
& V'&\\
}
\end{equation}
where $\alpha$ is the identity, so $\sF \times_{S'} V' \to (\sX \times_{S} S') \times_{S'} V'$ is an isomorphism. The locus $U'\subset S'$ over which $\sF \to \sX\times_{S}S'$ is an isomorphism, is Zariski open (and nonempty as $V' \subset U'$).

\begin{step}
Final step of the proof.
\end{step}

We also apply the preceding steps swapping $Y$ and $Z$, and thus obtain for $i=Y$ (resp. $i=Z$) finite morphisms $S_i \to S$ and nonempty Zariski open subsets $U_i \subset S_i$ such that the following hold:
\begin{enumerate}
	\item Let $D_i$ be the unique component of $\scrD(\sX/S)$ containing all points corresponding to subspaces of the form $Y(z)$ (resp. $Z(y)$).  Let $D_i^\nu \to D_i$ be the normalization. Then $D_i^\nu \to S_i \to S$ is the Stein factorization.
	\item Let $\sF_i \to D_i^\nu$ is the pullback of the universal family to the normalization. Then the morphism $\sF_i \to \sX \times_S S_i$ is an isomorphism over $U_i$. 
	\item The inclusion $V\subset S$ lifts to $V_i\subset S_i$ and the canonical map $\sZ_V \to D^\nu_Y\times_{S_Y} V_Y $ (resp. $\sY_V\to D^\nu_Z\times_{S_Z}V_Z$) is an isomorphism over $V$. 
	\end{enumerate}
Passing to $S':=S_Y\times_S S_Z$ and the intersection $U'$ of the preimages of $U_i$ under $S' \to S_i$ we obtain morphisms $\sF_i \times_{S_i} S' \to \sX \times_S S'$ whose restrictions to $U'$ are isomorphisms. In particular, we have projections $\sX\times_S U' \to D^\nu_i \times_{S_i} U'$.  Set $\sY=D_Z^\nu\times_{S_Z}U'$ and $\sZ=D_Y^\nu\times_{S_Y}U'$.  The product morphism
\begin{equation}\label{eq product decomposition morphism}
\sX\times_S U' \to \sY \times_{U'}  \sZ
\end{equation}
is an isomorphism when restricted over $V' \subset U'$. Note that
\[
\left(\sY \times_{U'} \sZ\right) \times_{S'} V' \isom \sY_{V}\times_V\sZ_V
\]
by construction.  By shrinking $U'$ we may assume that \eqref{eq product decomposition morphism} is an isomorphism and that $\sY,\sZ/U'$ are locally trivial, as the former is Zariski open and the latter Zariski closed (by Theorem \ref{theorem flenner kosarew}).
\end{proof}

\begin{remark}\label{remark cycle space meromorphic}
The proof actually shows a little more, namely that $\sX\times_S S'$ is bimeromorphic to a fiber product whose factors are locally trivial families over $U'\subset S'$.
\end{remark}

Let us record another consequence of the proof of Lemma~\ref{lemma decomposition zariski open}, mainly to set the notation for later use.

\begin{corollary}\label{cor limit leaves}
Let $X$ be a normal compact Kähler variety with rational singularities.  Suppose we have a locally trivial family $\sX\to \Delta$ with special fiber $X$ such that $\sX^*:=\sX\times_\Delta\Delta^*=\sY^*\times_{\Delta^*}\sZ^*$ where $\sY^*,\sZ^*/\Delta^*$ are locally trivial.  Then up to replacing $\Delta$ by a relatively compact open subset containing the origin there is a reduced\footnote{Note that $D_Y$ being reduced is equivalent to $\mathscr{D}(\sX/\Delta)$ being reduced at the generic point of $D_Y$. For the applications, this is irrelevant and we could also just have used the reduced structure.} component $D_Y\subset \mathscr{D}(\sX/\Delta)$ which is proper over $\Delta$ such that, setting $D_Y^*:=D_Y\times_\Delta\Delta^*$ and $\sF_Y^*:=\sF_Y\times_\Delta\Delta^*$ the restriction of the universal family $\sF_Y\subset \sX\times_{\Delta}D_Y$, the graph of the second projection $\sX^*\subset \sX^*\times_{\Delta^*}\sZ^*$ is isomorphic as a family of subspaces of $\sX^*$ over $\sZ^*$ to $\sF_Y^*\subset \sX^*\times_{\Delta^*} D_Y^*$ over $D_Y^*$.  Likewise for $D_Z,\sF_Z$.
\end{corollary}

\begin{proof}  By Corollary~\ref{corollary properness douady}, we can shrink $\Delta$ and find a component $D_Y$ of $\mathscr{D}(\sX/\Delta)$ which contains the generic fibers $Y_t\times\{z\}$ and which is proper over $\Delta$.  As in the proof of Lemma~\ref{lemma decomposition zariski open}, the projection $\sX^*\to\sZ^*$ yields a bijective classifying map $\sZ^*\to D_Y^*$ and $D_Y$ is smooth at points $Y_t\times\{z\}$ if $z\in \sZ^*$. In particular, the reducedness claim follows. The resulting pullback map $\sX^*\to\sF_Y^*$ is in fact an isomorphism, as the evaluation map $\sF_Y^*\to\sX^*$ provides an inverse.  It follows that the Stein factorization of $\sF_Y^*\to D_Y^*$ is identified with $\sX^*\to\sZ^*\to D_Y^*$, but as $\sX^*\to\sZ^*$ is faithfully flat and $\sF_Y^*\to D_Y^*$ is flat, it follows that $\sZ^*\to D_Y^*$ is flat hence an isomorphism.  
\end{proof}

\section{Splittings of relative tangent sheaves}\label{section tangent splitting}
In this section we show that given a product decomposition on the general fiber of a locally trivial family of $K$-trivial varieties, the induced splitting of the tangent bundle extends to the special fiber.

\begin{proposition}\label{proposition relative splitting}Let $\pi:\sX\to\Delta$ be a locally trivial family of $K$-trivial K\"ahler varieties with rational singularities.  Let $X$ be the special fiber.  Assume we have a product decomposition $\sX^*:=\sX\times_{\Delta}\Delta^*=\sY^*\times_{\Delta^*}\sZ^*$ for locally trivial families $\pi_1:\sY^*\to\Delta$ and $\pi_2:\sZ^*\to \Delta^*$.  Then there is a splitting
\[T_{\sX/\Delta}=A\oplus B\]
into foliations such that $A|_{\sX^*}=\pi_1^*T_{\sY^*/\Delta^*}$ and $B|_{\sY^*}=\pi_2^*T_{\sZ^*/\Delta^*}$ as subsheaves of $T_{\sX^*/\Delta^*}$. 

\end{proposition}

Before the proof of Proposition \ref{proposition relative splitting} we make some preliminary remarks.  We start by observing that there is a limit K\"unneth decomposition on the level of cohomology.

\begin{lemma}\label{lemma limit kunneth} Assume the setup of Proposition \ref{proposition relative splitting}. Let $j:\Delta^*\to\Delta$ be the inclusion.  Then we have decompositions of local systems 
\begin{equation}\label{eq kunneth}R^k\pi_*\Q_{\sX}\cong \bigoplus_{r+s=k}j_*R^r\pi_{1*}\Q_{\sY^*}\otimes_\Q j_*R^s\pi_{2*}\Q_{\sZ^*}\end{equation}
extending the K\"unneth decompositions over $\Delta^*$.
\end{lemma}
\begin{proof}As $\pi$ is topologically trivial, each $R^k\pi_*\Q_\sX$ has no local monodromy, and it follows that each $R^k\pi_{1*}\Q_{\sY^*}$ and $R^k\pi_{2*}\Q_{\sZ^*}$ also has no monodromy (since for instance $\pi_{1*}\Q_{\sY^*}$ and $\pi_{2*}\Q_{\sZ^*}$ have no monodromy).
\end{proof}
Note that the isomorphisms preserve the integral structure---that is, the torsion-free quotients of cohomology with integral coefficients---and that we also have the corresponding decomposition on the level of homology.  The decompositions are also compatible with cup and cap products.

Next we upgrade Lemma \ref{lemma limit kunneth} to a K\"unneth decomposition of variations of mixed Hodge structures.

\begin{lemma}Let $f:\sX\to S$ be a locally trivially family over a smooth base $S$.  Then for all $k$, $R^kf_*\Q_\sX$ underlies a variation of mixed Hodge structures
\end{lemma}
\begin{proof}

Note that because $S$ is smooth, the normalization of $\sX$ is locally trivial and specializes to the normalization fiberwise.  Following for instance the proof of \cite[Theorem~5.26]{PS08}, the family $\sX\to S$ then admits a locally trivial semisimplicial resolution over $S$ by Corollary~\ref{corollary resolutions}, and this is enough.  
\end{proof}
\begin{corollary}\label{corollary mixed kunneth}In the setup of Proposition \ref{proposition relative splitting}, the Hodge filtrations of $R^r\pi_{1*}\C_{\sY^*}\otimes_\C\sO_{\Delta^*}$ and $R^s\pi_{2*}\C_{\sZ^*}\otimes_\C\sO_{\Delta^*}$ extend so that \eqref{eq kunneth} holds as variations of mixed Hodge structures.  Moreover, the cup-product maps are morphisms of variations of Hodge structures.
\end{corollary}
\begin{proof}Let $U_k,V_r,W_s$ be the special fibers of $R^k\pi_*\C_\sX, j_*R^r\pi_{1*}\C_{\sY^*},j_*R^s\pi_{2*}\C_{\sZ^*}$, respectively.  Now denoting by $\mathrm{Fl}$ the appropriate flag varieties, the period map $\Delta\to\prod_k \mathrm{Fl}(U_k)$ associated to $R^\bullet\pi_*\C_{\sX}$ maps $\Delta^*$ to the image of the closed embedding of $\prod_k\mathrm{Fl}(V_k)\times\prod_k\mathrm{Fl}(W_k)$ via taking the graded tensor product.  It therefore maps $\Delta$ into $\prod_k\mathrm{Fl}(V_k)\times\prod_k\mathrm{Fl}(W_k)$, and the first claim follows.  The second claim is obvious as it is true fiber-wise.
\end{proof}

For simplicity, in the following we denote the Hodge filtration on $R^kf_*\C_\sX\otimes_\C\sO_\Delta$ (resp. $j_*R^rg_*\C_{\sY^*}\otimes_\C\sO_\Delta$, $j_*R^sh_*\C_{\sZ^*}\otimes_\C\sO_\Delta$) by $F^\bullet R^kf_*\C_\sX$ (resp. $F^\bullet j_*R^rg_*\C_{\sY^*}$, $F^\bullet j_*R^sh_*\C_{\sZ^*}$).

\begin{proof}[Proof of Proposition \ref{proposition relative splitting}]  Let $m=\dim\sY^*-1$ and $n=\dim \sZ^*-1$.  By Corollary~\ref{corollary resolutions} there is a locally trivial resolution $f:\wt\sX\to\sX$.  By Kebekus--Schnell \cite[Corollary~1.8]{KS18} we have $f_*\Omega^p_{\wt\sX/\Delta}=\Omega^{[p]}_{\sX/\Delta}$ via the natural map, and likewise for $\sY^*$ and $\sZ^*$.  In particular, we naturally identify $F^{p}R^{p}\pi_*\C_{\sX}= \Omega^{[p]}_{\sX/\Delta}$ for each $p$.  Moreover, both $F^mR^m\pi_{1*}\C_{\sY^*}= \Omega^{[m]}_{\sY^*/\Delta^*}$ and $F^nR^n\pi_{2*}\C_{\sZ^*}=\Omega^{[n]}_{\sZ^*/\Delta^*}$ are line bundles as $\sY^*$ and $\sZ^*$ are both families of $K$-trivial varieties.  From Corollary \ref{corollary mixed kunneth} it follows that $F^mj_*R^m\pi_{1*}\C_{\sY^*}$ is a vector subbundle of $F^mR^m\pi_*\C_\sX$, and so there is a $\tau_Y\in H^0(\sX,\Omega^{[m]}_{\sX/\Delta})$ which specializes to the pullback of a nonzero top-dimensional reflexive form on the fibers of $\sY^*$ and which is nonzero on the special fiber of $\sX$.  Likewise for $\tau_Z\in H^0(\sX,\Omega^{[n]}_{\sX/\Delta})$ and~$\sZ^*$.  

We claim that $A=\rad(\tau_Z)$ and $B=\rad(\tau_Y)$ provides the desired splitting.  Note that the radical can only increase in rank under specialization, and that $A$ and $B$ are generically complementary.  The claim therefore follows provided $\tau_Y\wedge\tau_Z$ is nonzero on the special fiber.  But Corollary \ref{corollary mixed kunneth} implies the K\"unneth decomposition provides us with an identification
\[F^{m+n}R^{m+n}\pi_*\C_\sX\cong F^mj_*R^m\pi_{1*}\C_{\sY^*}\otimes F^nj_*R^n\pi_{2*}\C_{\sZ^*}\]
where $\tau_Y\wedge\tau_Z$ on the left-hand side is identified with $\tau_Y\otimes\tau_Z$ on the right-hand side, and the claim follows.
\end{proof}

\section{Splittings of relative Kähler--Einstein metrics}\label{section kaehler einstein}

By the results of the last section, given a locally trivial family which is generically a product, there is both a limit splitting of the cohomology of the special fiber and a limit splitting of the tangent sheaf. We show in this section that the decomposition of the Kähler--Einstein metrics carries over to the limit.

\subsection{Forms and currents on singular spaces}\label{section forms and currents}

The references for this section are  \cite{Dem85} or \cite[\S 4.6.1]{BEG}. Let $X$ be normal complex space. We introduce the following sheaves on $X$: 
\begin{itemize}
\item $L^{1}_X$ is the sheaf of locally integrable real-valued functions on $X$, 
\item $\mathcal C^{\infty}_X$ is the sheaf of real-valued functions which are locally the restriction of a smooth function under a local embedding $X\underset{\rm loc}{\hookrightarrow} \mathbb C^N$ ,
\item $\mathrm{PH}_X \subset \mathcal C^{\infty}_X$ is the subsheaf made of pluriharmonic functions. A pluriharmonic function can be equivalently defined as being locally the real part of a holomorphic function (which by definition comes from a local embedding) or a function in the kernel of the $dd^c$ operator.  
\end{itemize}
We recall the following definitions. 

\begin{enumerate}[label=$\--$]
	\item A $(1,1)$-form on $X$ is a smooth $(1,1)$-form on $X^{\rm reg}$ that extends smoothly under local embeddings.
	\item A $(1,1)$-form is said to have {\it local potentials} if it arises as a global section of $\Cx/\PHx$ via the $dd^c$ operator.  In particular, it is closed.
	\item A closed $(1,1)$-current {\it with local potentials} is a global section of  $L^1_X/\PHx$. 
	\item A \emph{plurisubharmonic} function (psh for short) on $X$ is a locally integrable function on $X$ which is the restriction of a psh function under a local embedding. If $\theta$ is a $(1,1)$-form with local potentials on $X$, a function $\varphi$ on $X$ is said to be \emph{$\theta$-psh} if it is locally the sum of a smooth and a psh function and it satisfies $\theta+dd^c \varphi \ge 0$ weakly. We denote by $\mathrm{PSH}(X,\theta)$ the set of $\theta$-psh functions on $X$. If we do not specify $\theta$, we also speak of \emph{quasi-psh functions}.
	\item A \emph{Kähler metric} on $X$ is a $(1,1)$-form with local potentials whose local potentials are strictly plurisubharmonic, i.e. they are the restriction of a (smooth) strictly psh function under a local embedding. 

\end{enumerate}

\begin{remark}
\label{pull back}
A closed $(1,1)$-current $T$ with local potentials on $X$ enjoys the important property that given a surjective morphism $f: Y \to X$ of complex varieties, one can define its pull-back by $f^*T$ by lifting its local potentials. The current $f^*T$ is again closed, of type $(1,1)$ and admits local potentials on $Y$. Moreover, $f^*T$ is positive if and only if $T$ is positive. 
\end{remark}

We have the following  exact sequences
\begin{equation}
\label{DR}
0 \longrightarrow \PHx \longrightarrow L^1_X \longrightarrow L^1_X/\PHx \longrightarrow 0
\end{equation}
and similarly with $\Cx$ in place of $L^1_X$, as well as 
\begin{equation}
\label{PH}
0 \longrightarrow \underline{\mathbb{R}}_X \overset{i\cdot}{\longrightarrow} \mathcal O_X \overset{\mathrm{Re}(\cdot)}{\longrightarrow} \PHx \longrightarrow 0
\end{equation}
Composing the exterior differential map $d:\mathcal O_X\to \Omega_X^1$ with the isomorphism $\PHx\isom \mathcal O_X /\underline{\mathbb R}_X$, and passing to cohomology, we get a natural map 
\begin{equation}
\label{diff map}
d:H^1(X,\PHx)\to H^1(X,\Omega_X^1).
\end{equation}
Next, the exact sequences \eqref{DR}-\eqref{PH} above yield exact sequences in cohomology
\begin{equation}\label{eq coh one}
H^0(X,L^1_X/\PHx)\overset{[\cdot]}{\longrightarrow} H^1(X,\PHx) \longrightarrow 0
\end{equation}
and
\begin{equation}
\label{eq coh two}
H^1(X,\PHx) \overset{\alpha}{\longrightarrow} H^2(X,\mathbb R) \overset{\beta}{\longrightarrow} H^2(X,\mathcal O_X)
\end{equation}
A class in $H^1(X,\mathrm{PH}_X)$ (resp. in $H^2(X,\mathbb R)$) is said to be Kähler if it is the image of a Kähler metric under the map $[\cdot]$ in \eqref{eq coh one} (resp. under the map $\alpha \circ [\cdot]$). According to \cite[Proposition~2.8]{BL18}, the Kähler class ends up in $H^{1,1}(X,\R):=F^1H^2(X,\C)\cap H^2(X,\R)$.
The following elementary result is very useful.

\begin{proposition}
\label{H1PH}
Let $X$ be a compact, normal variety of class $\scrC$ with rational singularities. Then $\alpha$ is injective and $\beta$ is surjective; i.e. we have an exact sequence
\begin{equation}\label{eq h1 pluriharmonic sequence}
0\longrightarrow H^1(X,\PHx) \overset{\alpha}{\longrightarrow} H^2(X,\mathbb R) \overset{\beta}{\longrightarrow} H^2(X,\mathcal O_X) \longrightarrow 0.
\end{equation}
\end{proposition}
\begin{proof}
The injectivity of $\alpha$ is proved in \cite[Remark~3.2 (2)]{GK20}. Surjectivity of $\beta$ follows because the Hodge structure on $H^2(X,\C)$ is pure and $\beta$ can be identified with the projection on the $(0,2)$-part. Clearly, for $\alpha \in H^{0,2}(X)$ the class $\alpha + \bar \alpha$ is real.
\end{proof}

\subsection{Continuity of the relative KE metric}
For the remainder of this section, we work with the following setup. 

\begin{set}\label{setting relativeKE}
Let $\sX$ be a normal complex space with canonical singularities equipped with a proper, holomorphic, surjective map  $\pi: \sX \to \Delta$ and a smooth hermitian positive definite $(1,1)$-form $\theta$ such that, setting $X_t:=\pi^{-1}(t)$ and $X=X_0$, one has
\begin{enumerate}[label=(\roman*)]
\item $\omega_{\sX/\Delta} \isom \sO_{\sX}$.
\item The map $\pi$ is locally trivial and $H^1(X_0,\mathbb R)=0$. 
\item The restriction $\theta_t:=\theta|_{X_t}$ is a Kähler form.
\item There is a splitting $\sX\times_\Delta\Delta^*\isom \sY^*\times_{\Delta^*}\sZ^*$ with $\pi_1:\sY^*\to\Delta^*$ and $\pi_2:\sZ^*\to\Delta^*$ locally trivial.
\item The relative tangent sheaf splits: $T_{\sX/\Delta}=A\oplus B$ and $A|_{\sX^*}=\pi_1^*T_{\sY^*/\Delta^*}$ and $B|_{\sY^*}=\pi_2^*T_{\sZ^*/\Delta^*}$ as subsheaves of $T_{\sX^*/\Delta^*}$.
\end{enumerate}
As mentioned before, this in particular implies that $\pi$ is topologically, even real analytically trivial by \cite[Proposition~6.1]{AV19}.
\end{set}


In Setting~\ref{setting relativeKE}, we can consider for any $t\in \Delta$ the unique singular Kähler--Einstein metric $\omega_t \in \{\theta_t\}$, provided by \cite[Theorem~A]{EGZ}. One can write $\omega_t$ in a unique way as
\[
\omega_t=\theta_t+dd^c \vphi_t
\]
where $\vphi_t$ is a $\theta_t$-psh function normalized by $\int_{X_t} \vphi_t \, \theta_t^n=0$. Moreover, it is known from {\it loc. cit.} that $\vphi_t \in L^{\infty}(X_t)$ and from \cite[Corollary~1.1]{Paun} that $\omega_t$ is a genuine Kähler form on $X_t^\reg$, the regular locus of $X_t$.  Up to shrinking $\Delta$, it follows from \cite[Theorem~F]{DGG20} that there exists a constant $C>0$ independent of $t\in \Delta$ such that 
\begin{equation}
\label{Linfty}
\sup_{X_t} \vphi_t-\inf_{X_t} \vphi_t \le C.
\end{equation}
Actually, the quoted result requires $\theta$ to be relatively Kähler (hence globally Kähler on  $X$, possibly after shrinking $\Delta$). However, taking a closer look at the proof reveals that the only (truly) global argument appears in Theorem~2.9 in {\it ibid.}, where one needs to control away from $0$ and $+\infty$ the masses $\int_{X_t} (\theta_t+dd^c\psi_t)\wedge \theta_t^{n-1}$ uniformly for any $\theta_t$-psh function $\psi_t$ and any $t\in \Delta$. In our situation, this property is clearly satisfied since $\theta_t$ is Kähler (hence closed) and $\int_{X_t} \theta_t^n$ is a continuous function of $t\in \Delta$.

Let us also observe that the normalization condition $\int_{X_t} \vphi_t \, \theta_t^n=0$ implies that $\sup_{X_t} \vphi_t \ge 0$ and $\inf_{X_t} \vphi_t \le 0$. The bound on the oscillation \eqref{Linfty} thus implies that with the same constant $C$, we have
 \begin{equation}
\label{Linfty2}
-C \le  \vphi_t \le C
\end{equation}
for all $t\in \Delta$.

Up to shrinking $\Delta$, by local triviality one can assume that there exists a finite covering $(U_\alpha)_{\alpha \in \Gamma}$ of $\sX$ by Euclidean open subsets such that if we set $U_{t,\alpha}:=U_\alpha \cap X_t$ for $t\in \Delta$, there exist biholomorphisms 
$F_\alpha: U_{\alpha,0}\times \Delta \to U_\alpha $
 over $\Delta$ and such that $F_\alpha|_{U_{\alpha,0}\times \{0\}}$ coincides with $\id_{U_{\alpha,0}}$ via the natural identification $U_{\alpha,0}\times \{0\} \isom U_{\alpha,0}$. In other words, restriction to a fiber gives a family of biholomorphisms
 \begin{equation}
 \label{local biholo}
 F_{\alpha,t}: U_{\alpha,0} \to U_{\alpha,t} 
 \end{equation}
depending holomorphically on $t$ such that $F_{\alpha,t}$ converges to $\id_{U_{\alpha,0}}$ when $t\to 0$. In the following, we will replace each $U_{\alpha}$ by a slightly smaller open set so that we can assume that the biholomorphisms $F_{\alpha,t}$ extend to a neighborhood of $\partial U_{\alpha, 0}$.

\begin{proposition}\label{proposition convergence}
For any $\alpha \in \Gamma$, the currents $F_{\alpha,t}^*\omega_t|_{U_{\alpha,t}}$ converge to $\omega_0|_{U_{\alpha,0}}$ when $t\to 0$, weakly on $U_{\alpha,0}$ and locally smoothly on the regular locus of that set. 
\end{proposition}

\begin{proof}
Let us write $F_{\alpha,t}^*\omega_t = \theta_{\alpha,t}+dd^c \psi_{\alpha,t}$ on $U_{\alpha,0}$ where $\theta_{\alpha, t}:=F_{\alpha,t}^*\theta_t$ and $\psi_{\alpha,t}:=\vphi_t \circ  F_{\alpha,t}$. First observe that $\psi_{\alpha,t}$ is a  $\theta_{\alpha,t}$-psh function.  We are going to show at the same time the following three assertions, which altogether prove the proposition. 
\begin{itemize}
\item The family $(\psi_{\alpha, t})$ is precompact in the $L^1_{\rm loc}$ topology.
\item Any weakly convergent subsequence $(\psi_{\alpha, t_j})$ actually converges locally smoothly on $U_{\alpha, 0}^{\rm reg}$.
\item The only cluster value for $(\psi_{\alpha, t})$ when $t\to 0$ is ${\varphi_0}|_{U_{\alpha,0}}$.
\end{itemize}
Note that even though the results above are local (i.e. on $U_{\alpha, 0}$), we need a global argument to identify the sequential limits and obtain the last item. 

The first item follows directly from classical pluripotential theory given the estimate
\begin{equation}
\label{Linfty3}
\| \psi_{\alpha, t}\|_{L^{\infty}(U_{\alpha,0})} \le C
\end{equation}
 infered by \eqref{Linfty2} and the fact that $\theta_{\alpha,t}$ converges (smoothly) to $\theta_0$ on $U_{\alpha_0}$.

We now proceed to prove the remaining two assertions. Let us first observe that if $\alpha,\beta$ are two indices such that $U_{\alpha\beta, 0}:=U_{\alpha, 0}\cap U_{\beta, 0} \neq \emptyset$ and $\psi_{\alpha, t_j}$ converges to a quasi-psh function $\psi_{\alpha,0}$ on $U_{\alpha, 0}$, then any sequential limit $\psi_{\beta,t_j'}\to \psi_{\beta,0}$ for a subsequence $(t_j')$ of $(t_j)$ will satisfy $ \psi_{\alpha,0}=\psi_{\beta, 0}$ on $U_{\alpha\beta,0}$. 
 This follows from the fact that $(F_{\alpha,t}|_{U_{\alpha\beta,0}})\circ (F_{\beta,t}|_{U_{\alpha\beta,0}})^{-1}$ converges to $\id_{U_{\alpha\beta,0}}$ when $t\to 0$. Since the set $\Gamma$ of indices is finite, we can iterate this construction and out of any sequence $t_j \to 0$, one can find a subsequence $(t_j')$ and a function $\psi_0\in \mathrm{PSH}(X_0,\theta_0)\cap L^{\infty}(X_0) $ such that for any $\alpha \in \Gamma$, the current $F_{\alpha, t_j'}^*\omega_{t_j'}$ converges to $\theta_0+dd^c \psi_0$ on $U_{\alpha, 0}$ in the weak topology.

In order to finish the proof of the proposition, we will show the following.
\begin{itemize}
\item  The weak convergence $F_{\alpha, t_j'}^*\omega_{t_j'}\to \theta_0+dd^c \psi_0$ is locally smooth over $X_0^\reg$;
\item  The current $\theta_0+dd^c \psi_0$ coincides with $\omega_0$.
\end{itemize}
Since $\pi$ is locally trivial, the functorial resolutions $p_t:\wt X_t\to X_t$ can be patched together to induce a simultaneous resolution $p:\wt \sX\to \sX$ that restricts to $p_t$ over $X_t$ for any $t\in \Delta$, cf. \cite[Lemma~4.8]{BL18}. As the simultaneous resolution is a projective morphism, one can reproduce verbatim the argument in \cite[\S~3]{Paun} (see also \cite[Appendix~B]{BBEGZ}) relying on Tsuji's trick \cite{Tsuji88} to get laplacian estimates 
\begin{equation}
\label{laplacian}
(\tr_{\theta_t} \omega_t)|_{K_t} \le C_K 
\end{equation}
for any compact subset $K\Subset \sX^\reg$ where $K_t:=K\cap X_t$ and $C_K>0$ is a constant independent of $t$. From \eqref{laplacian}, one can get higher order estimates 
\begin{equation}
\label{Ck}
\|\varphi_t\|_{\mathcal C^k(K_t)}  \le C_{K,k} 
\end{equation}
for any integer $k\ge 0$ using standard results (Evans-Krylov and Schauder estimates, cf. e.g. \cite[\textsection 5.6]{Bl2} and \cite[\textsection 2.3]{SzGTM}). 

Now, let $\Omega \in H^0(\sX, \omega_{\sX/\Delta})$ be a trivialization. We know from e.g. the proof of \cite[Theorem~6.1]{DGG20} that there exists a constant $c_t>0$ uniformly bounded away from $0$ and $+\infty$ when $t\in \Delta$ such that 
\begin{equation}\label{eq MAt}
(\theta_t+dd^c \varphi_t)^n = c_t \cdot i^{n^2} \Omega_t \wedge \overline \Omega_t \quad \mbox{on } X_t.
\end{equation} 
By \eqref{Ck}, the function $\psi_{\alpha,t}$ is locally bounded for any $\mathcal C^k$-norm on $U_{\alpha, 0}^\reg:=U_{\alpha,0}\cap X_0^\reg$. Therefore, the Arzela--Ascoli theorem guarantees that one can find a subsequence $(t_j'')$ of $(t_j')$ such that $\psi_{\alpha, t_j''}$ converges locally smoothly on $U_{\alpha, 0}^\reg$. Since $\psi_{\alpha, t_j''}$ already converges weakly to $\psi_0$ on $U_{\alpha, 0}$, we have $\psi_{\alpha,0}|_{U_{\alpha,0}^\reg}=\psi_0|_{U_{\alpha, 0}^\reg}$. This shows that $\psi_{\alpha,t}$ converges locally smoothly to $\psi_0$ on $U_{\alpha, 0}^\reg$ when $t\to 0$. Applying $F_{\alpha, t_j''}^*$ to \eqref{eq MAt} and passing to the limit on each $U_{\alpha, 0}^\reg$, one finds that 
\begin{equation}\label{eq MAt2}
(\theta_0+dd^c \psi_0)^n  = c_0 \cdot i^{n^2} \Omega_0 \wedge \overline \Omega_0 \quad \mbox{on } X_0^\reg.
\end{equation} 
As $\psi_0 \in L^{\infty}(X_0)$, the uniqueness of the Kähler--Einstein metric \cite[Proposition~1.4]{EGZ} guarantees that $\theta_0+dd^c \psi_0=\omega_0$. 
 \end{proof}

\subsection{Decomposition of the KE metric on $X_0$}\label{sec:dec}

We now prove the splitting of the K\"ahler-Einstein metric. 
\begin{proposition}
\label{proposition decompKE}
In Setting~\ref{setting relativeKE}, one can decompose the Kähler-Einstein metric $\om_t$ on $X_t^\reg$ as 
\begin{equation}
\label{eq decompKE}
\om_t=\omega_{A_t}+\omega_{B_t}
\end{equation}
where
\begin{enumerate}
\item\label{list 1} The forms $\omega_{A_t},\omega_{B_t}$ are closed, semipositive smooth $(1,1)$-forms on $X_t^\reg$.
\item\label{list 2} One has $\mathrm{ker} \, \omega_{A_t} =  B_t$ and  $\mathrm{ker} \, \omega_{B_t} = A_t$.
\item\label{list 3} When $t\to 0$, we have local smooth convergence $\omega_{A_t}\to \omega_{A_0}$ (resp. $\omega_{B_t}\to \omega_{B_0}$) on $X_0^\reg$ under any local trivialization of the family. 
\item\label{list 4} If $t\neq 0$, then $\om_{A_t}$ (resp. $\omega_{B_t}$) is the pull-back of a KE metric on $Y_t$ (resp. $Z_t$). In particular, they extend to $X_t$ as positive currents with bounded local potentials satisfying the identity~\eqref{eq decompKE}.
\end{enumerate}
\end{proposition}

\begin{proof}\ 

\noindent
{\it Items \eqref{list 1} and \eqref{list 2}.} The decomposition of the tangent sheaf $T_{X_t}=A_t\oplus B_t$ induces a decomposition of $T_{X_t^\reg}$ into {\it parallel} subbundles with respect to $\omega_t$ by \cite[Theorem~8.1.2]{GGK19}. That result was stated for projective varieties but the proof of that particular result is purely analytic and does not rely on the projectivity assumption, cf. \cite[Remark~3.5]{CGGN}. 

This implies that the Kähler--Einstein metric $\omega_t$ splits canonically on $X_t^\reg$ as follows
\begin{equation*}
\omega_t=\omega_{A_t}+\omega_{B_t}
\end{equation*}
where $\omega_{A_t}$ (resp. $\omega_{B_t}$) is a smooth, closed semipositive $(1,1)$-form on $X_t^\reg$ whose kernel is $B_t|_{X_t^\reg}$ (resp $A_t|_{X_t^\reg}$) and which is positive definite in restriction to $A_t|_{X_t^\reg}$ (resp. $B_t|_{X_t^\reg}$). \\

\noindent
{\it Item \eqref{list 3}.} This is an easy consequence of the local smooth convergence of $\omega_t$ to $\omega_0$ on $X_0^\reg$ via the local biholomorphisms $F_{\alpha,t}$ shown in Proposition~\ref{proposition convergence} combined with the definition of $\omega_{A_t},\omega_{B_t}$ and the fact that the summand $A_t$ of $T_{X_t}$ coincides with the restriction of the subbundle $A\subset T_{\sX/\Delta}$ to $X_t^\reg$.\\

\noindent
{\it Item \eqref{list 4}.} From now on, we assume $t\neq 0$. We will need the following
 
 \begin{lemma}
 \label{product}
For any $t\in \Delta^*$, the natural map
 \[H^1(Y_t,\mathrm{PH}_{Y_t}) \oplus  H^1(Z_t,\mathrm{PH}_{Z_t}) \longrightarrow H^1(X_t,\mathrm{PH}_{X_t})\]
 is an isomorphism.
 \end{lemma}
 
\begin{proof}
Recall from Proposition~\ref{H1PH} that $H^1(X_t,\mathrm{PH}_{X_t})$ is the kernel of the natural map $H^2(X_t,\mathbb R)\to H^2(X_t, \mathcal O_{X_t})$. We claim that this map is given by the sum map $H^2(Y_t,\mathbb R)\oplus H^2(Z_t,\mathbb R)\to H^2(Y_t, \mathcal O_{X_t})\oplus H^2(Z_t, \mathcal O_{X_t})$, from which the lemma follows.
 
 Now, since $\pi$ is locally trivial and $H^1(X_0, \mathbb R)=0$, we have $H^1(X_t, \mathbb R)=0$. Since $X_t$ has rational singularities, this implies that $H^1(X_t, \mathcal O_{X_t})=0$ as well. The claim now follows from Künneth decomposition formula.
 \end{proof}

 Therefore, one can decompose $[\theta_t]\in H^{1}(X_t,\mathrm{PH}_{X_t})$ as 
 $$[\theta_t]=\mathrm{pr}_{t,1}^*\alpha_t+\mathrm{pr}_{t,2}^*\beta_t$$
  for some classes $\alpha_t$ and $\beta_t$ on $Y_t$ and $Z_t$ respectively. Since $[\theta_t]$ is Kähler, so are $\alpha_t$ and $\beta_t$.  By \cite{EGZ}, there exists a unique singular Kähler--Einstein metric $\omega_{Y_t}\in \alpha_t$ (resp. $\omega_{Z_t} \in \beta_t$). Since $\mathrm{pr}_{t,1}^*\omega_{Y_t}+\mathrm{pr}_{t,2}^*\omega_{Z_t} \in [\theta_t]$ is smooth, Kähler and Ricci-flat on $X_t^\reg$ and has bounded potentials, it coincides with $\omega_t$. In particular, one has $\omega_{A_t}=\mathrm{pr}_{t,1}^*\omega_{Y_t}$ and  $\omega_{B_t}=\mathrm{pr}_{t,1}^*\omega_{Z_t}$ on $X_t^\reg$, hence everywhere since none of these currents puts any mass on $X_t^{\rm sing}$. 
 \end{proof}

\subsection{Properties of the currents $\omega_{A_0}$ and $\omega_{B_0}$}
In this section, we fix a functorial, simultaneous resolution $p:\wt\sX\to  \sX$ inducing fiber-wise resolutions $p_t: \wt X_t\to X_t$, we set $\wt \pi:=\pi \circ p$. Since $p_*T_{\wt \sX}=T_{\sX}$ and $\sX\to \Delta$ is locally trivial, one can find a cover $(U_\alpha)$ of $\sX$ and holomorphic vector fields $v_\alpha$ on $U_\alpha$ such that 
\begin{itemize}
\item $\pi_*v_\alpha=\frac{\partial}{\partial t}$.
\item $v_\alpha=p_*\wt v_\alpha$ for some holomorphic vector field $\wt v_\alpha$ on $p^{-1}(U_\alpha)$. 
\end{itemize}
Using a partition of unity $(\chi_\alpha)$ subordinate to $(U_\alpha)$, one can construct a smooth vector field $v$ on $\sX$ (resp. $\wt v$ on $\wt \sX$) such that $p_*\wt v= v$ and $\pi_*v=\frac{\partial}{\partial t}$.  The flow of these vector fields yields diffeomorphisms $F: X_0\times \Delta \to \sX$ and $\widetilde F: \wt X_0\times \Delta \to \wt \sX$ commuting with $p$; i.e. $p \circ \wt F = F \circ (p_0 \times \id_{\Delta})$. With a slight abuse, we denote by $F_t$ (resp. $\wt F_t$) the restriction $F|_{X_0\times \{t\}}: X_0\to X_t$ (resp. $\wt F|_{\wt X_0\times \{t\}}: \wt X_0\to \wt X_t$). 

For $t\neq 0$, the currents $\omega_{A_t},\omega_{B_t}$ have local potentials by Proposition~\ref{proposition decompKE}, hence they induce cohomology classes $\alpha_t,\beta_t \in H^2(X_t, \mathbb R)$   such that 
\begin{equation}
\label{sum}
\alpha_t+\beta_t= [\theta_t].
\end{equation}
\begin{lemma}\label{lemma alpha}
With the notation above, there exist $\alpha_0, \beta_0\in H^2(X_0, \mathbb R)$ such that $F_t^*\alpha_t \to \alpha_0$ (resp. $F_t^*\beta_t \to \beta_0$) when $t\to 0$.
\end{lemma}

\begin{proof}
From Lemma \ref{lemma limit kunneth} we have $R^2\pi_*\R_\sX=j_*R^2\pi_{1*}\R_{\sY^*}\oplus j_*R^2\pi_{2*}\R_{\sZ^*}$, using (ii) of Setting~\ref{setting relativeKE}.  The sections $t\mapsto \alpha_t$ (resp. $t\mapsto \beta_t$) of the vector bundle $(R^2\pi_{1*}\R_{\sY^*})\otimes \mathcal C_{\Delta^*}^{\infty}$ (resp. $(R^2\pi_{2*}\R_{\sZ^*})\otimes \mathcal C_{\Delta^*}^{\infty}$) over $\Delta^*$ are such that their sum extends to a smooth section of $(R^2\pi_*\R_\sX)\otimes \mathcal C_{\Delta}^{\infty}$. As each section lives in a different direct summand, both extend separately across the origin, proving the claim.
\end{proof}

Another helpful result is the following 

\begin{lemma}\label{H11}
In Setting~\ref{setting relativeKE}, there is a sub-vector bundle $\mathcal H \subset R^2\pi_*{\mathbb R}_{\mathcal X} \otimes_\R \mathcal C^{\infty}_{\Delta}$ and the inclusion is fiberwise canonically identified with $H^1(X_t, \mathrm{PH}_{X_t})\subset H^2(X_t, \mathbb R)$.
\end{lemma}
\begin{proof}
By topological triviality and purity of the weight two Hodge structure, the family $\sX \to  \Delta$ gives rise to a variation of Hodge structures with underlying $\R$-local system $R^2\pi_*{\mathbb R}_{\mathcal X}$. We denote $\sV:= R^2\pi_*{\mathbb \R}_{\mathcal X} \tensor \sC^{\infty}_{\Delta}$ and by $F^p \subset \sV \tensor \C$ the $\mathcal{C}^\infty$-sections of the Hodge filtration. Then the base change of $\sV \to F^0/F^1$ to $t\in \Delta$ is the surjection $\beta$ from \eqref{eq h1 pluriharmonic sequence}, so if we define $\sH:=\ker\left(\sV \to F^0/F^1\right)$, we have canonical identifications $\sH_t = H^1(X_t,\PH_{X_t})$.
\end{proof}

\begin{proposition}
\label{boundedpotentials}
With the notation of Proposition~\ref{proposition decompKE}, then the currents $\omega_{A_0},\omega_{B_0}$ extend to positive currents on $X_0$ having local bounded potentials and satisfying~\eqref{eq decompKE} on $X_0$. 
\end{proposition}

\begin{proof}
Let us first prove that $\omega_{A_0},\omega_{B_0}$ can be extended across $X_0^{\rm sing}$ with local potentials. We prove the claim for $\omega_{A_0}$, since the argument for $\omega_{B_0}$ is entirely similar. First of all, since the singularities of $X_t$ are rational, $p$ realizes $R^2\pi_*\mathbb R_{\sX} \otimes \mathcal C^{\infty}_{\Delta}$ as a sub-vector bundle of $R^2\wt \pi_*\mathbb R_{\wt \sX} \otimes \mathcal C^{\infty}_{\Delta}$, see e.g. \cite[Lemma~2.1]{BL16}. In particular, Lemma~\ref{H11} provides a sub-vector bundle
\[p^*\mathcal H \hookrightarrow R^2\wt \pi_*\mathbb R_{\wt \sX} \otimes \mathcal C^{\infty}_{\Delta}\]
whose fiber at $t$ is $p^*H^1(X_t,\mathrm{PH}_{X_t}) \subset H^2(\wt X_t, \mathbb R)$. 

By Lemma~\ref{lemma alpha}, the section $t\mapsto p_t^*\alpha_t$ is a continuous section of $R^2\wt \pi_*\mathbb R_{\wt \sX} \otimes \mathcal C^{\infty}_{\Delta}$ whose image lies in $p^*\mathcal H$ over $\Delta^*$. It follows that the latter holds across the origin too. In particular, we get 
\begin{equation}
\label{alpha0}
p_0^*\alpha_0  \in p_0^*\mathcal H_0.
\end{equation}

Moreover, the arguments laid out in \cite[Appendix]{AH19} show that one can find a sequence $t_k\to 0$ and a closed, positive $(1,1)$-current $T$ on $\wt X_0$ such that $\wt F_{t_k}^*p_{t_k}^*\omega_{A_{t_k}}$ converges to $T$ weakly when $t_k\to 0$. Since $\wt X_0$ is a smooth, compact Kähler manifold, one can write $T=\wt \gamma+dd^c\wt  v$ where $\wt \gamma$ is a smooth, closed $(1,1)$-form on $\wt X_0$ and $\wt v\in L^1(\wt X_0)$. Since the cohomology class in $H^2(\wt X_0, \mathbb R)$ of a closed current depends continuously on it (with respect to the weak topology), we have $[\wt \gamma] \in p_0^*\mathcal H_0$. Said otherwise, the local potentials of $\wt \gamma$ actually come from $X_0$ via $p_0$ modulo a global function on $\wt X_0$, i.e. one can write $\wt \gamma = p_0^*\gamma+dd^c \wt w$ where $\gamma$ is a $(1,1)$-form on $X_0$ with potentials and $\wt w\in L^1(\wt X_0)$.  Set $\wt u:=\wt v+\wt w$ so that $T=p_0^*\gamma+dd^c \wt u$. 

Given a fiber $F$ of $p_0$, $\widetilde u|_F$ is either identically $-\infty$ or it is a psh function. Since $F$ is connected and compact, $u|_F$ has to be constant. That is, there exists a $\gamma$-psh function $u$ on $X_0$ such that $\wt u=\gamma^*u$ or, equivalently, $T=p_0^*(\gamma+dd^c u)$ on $\wt X_0$.

Now, recall from Proposition~\ref{proposition decompKE} that $F_t^* \omega_{A_t}\to \omega_{A_0}$ locally smoothly on $X_{0}^{\rm reg}$ when $t\to 0$. By uniqueness of the limit, this shows that $T$ agrees with $p_0^*\omega_{A_0}$ on $\wt X_0 \setminus \mathrm{Exc}(p_0)$. Therefore, we have $\omega_{A_0}=\gamma+dd^c u$ on $X_0^{\rm reg}$. This shows that $\omega_{A_0}$ extends to a closed positive $(1,1)$-current with local potentials on $X_0$. 

In the following, the extensions just constructed to $X_0$ will still be denoted by $\omega_{A_0}$ (resp. $\omega_{B_0}$), and we aim to prove that their local potentials are bounded. Pick a Stein open subset $U\subset X_0$ where $\omega_{A_0}|_{U}=dd^c v_1$ (resp. $\omega_{B_0}|_{U}=dd^c v_2$)  and $\omega_0|_U=dd^c v$ where $v,v_i\in \mathrm{PSH}(U)$ and $v$ is bounded on $U$, say $\|v\|_{L^{\infty}(U)} \le C_1$.  Up to shrinking $U$ slightly and subtracting a constant, one can assume that the upper semicontinuous functions $v_i$ are nonpositive on $U$. The function $v-(v_1+v_2)$ coincides almost everywhere with a pluriharmonic function $h$ on $U$. Up to shrinking $U$, one can assume that $h$ is bounded; i.e. $\|h\|_{L^{\infty}(U)} \le C_2$. Then, one has $0 \ge v_i \ge -(C_1+C_2)$ almost everywhere on $U$, hence everywhere on $U$ since $v_i$ is psh.  Finally, the closed $(1,1)$-current $\om_0-(\omega_{A_0}+\omega_{B_0})$ is supported on $X_0^{\rm sing}$ while it has local bounded potentials, so it vanishes everywhere thanks to the support theorem applied on $\widetilde X$, which is legitimate by Remark~\ref{pull back}. Indeed, $p_0^*(\om_0-(\omega_{A_0}+\omega_{B_0}))$ would have to be a current of integration along a divisor, which violates the boundedness of its potentials unless it vanishes identically. 
\end{proof}

\section{Splittings of locally trivial families}\label{section family splitting}

In this section, we show using the results of the previous section that the limit tangent splitting of Section \ref{section tangent splitting} induces a product decomposition.  The key technical result of the previous section is the existence of a splitting of the K\"ahler--Einstein metric as a positive current with bounded potentials. 

The {\it boundedness} of the local potentials of a closed, positive $(1,1)$-current $T$ on $X$ allows us to do two things that we cannot do with closed positive $(1,1)$-currents in general, even when they admit local potentials. 
\begin{enumerate}
\item One can define the Bedford-Taylor Monge-Ampère operator $T^n$, which is a positive measure satisfying $\int_X T^n = [T]^n$ where the right-hand side is computed in cohomology, cf. \cite{BT76,BedT82} or \cite[\S~2]{Dem85}.
\item If $Z\subset X$ is a closed analytic subvariety of $X$, then one can define $T|_Z:=\theta|_Z+dd^c u|_Z$, which is again a closed, positive $(1,1)$-current with bounded local potentials on $Z$. 
\end{enumerate}

We assume we are in Setting~\ref{setting relativeKE} and adopt the notation of Corollary \ref{cor limit leaves}.  Thus, after shrinking $\Delta$ we have the reduced irreducible components $D_Y,D_Z$ of $\mathscr{D}(\sX/\Delta)$ which are proper over $\Delta$, generically parametrize fibers $Y_t\times\{z\}$ and $\{y\}\times Z_t$, respectively, and furthermore $D_Y\times_{\Delta}\Delta^*=\sZ^*$ and $D_Z\times_{\Delta}\Delta^*=\sY^*$.  Let $(D_Y)_0, (D_Z)_0$ be the fibers over $0$ (endowed with the reduced structure), and set $m=\dim \sY^*-1$, $n=\dim \sZ^*-1$.  Let $\sF_Y,\sF_Z$ be the restrictions of the universal families to $D_Y,D_Z$. 

The map $\sF_Z\overset{e}{\longrightarrow}\sX$ is an isomorphism over $\Delta^*$, hence it is surjective by properness. For dimensional reasons, $X\not \subset e(\mathrm{Exc}(e))$ so one can consider the proper transform $(\sF_Z)_0^+$ of $X$ by $e$. This is the unique component of the special fiber $(\sF_Z)_0$ for which the map to $X$ is surjective and generically one-to-one. We endow $(\sF_Z)_0^+$ with the reduced structure\footnote{One can actually prove that $(\sF_Z)_0^+$ is reduced with its given structure, i.e. the one for which the inclusion in the special fiber is an isomorphism at the generic point of $(\sF_Z)_0^+$. Just notice that $\sF_Z$ is reduced by \cite[Lemma~1.4]{Fuj78} and that $e$ is an isomorphism at the generic point of $(\sF_Z)_0^+$. We will however not need this.}.

 We have a diagram
\begin{equation}\label{eq ev}\begin{tikzcd}
(\sF_Z)_0^+\ar[d,"p"']\ar[r,"e"]&X.\\
(D_Z)_0
\end{tikzcd}\end{equation}
Let $(D_Z)_0^+=p((\sF_Z)_0^+)$.  Note that for any $[V_0]\in (D_Z)_0^+$, we have that $V_0\cap X^\reg$ is either empty or tangent to $B|_{X^\reg}$, since by choosing the germ of a curve $C$ in $D_Z$ passing through $[V_0]$ and dominating $\Delta$, we obtain a flat family $\sV/C$ specializing to $V_0$, embedded as $\sV\subset \sX':=\sX\times_\Delta C$, and generically tangent to $B$.  Likewise we define $(\sF_Y)_0^+$ and $(D_Y)_0^+$.
  
\begin{lemma}\label{lemma irred}Let $X$ be as in Setting~\ref{setting relativeKE}.  Then for any point $[V_0]\in (D_Z)_0^+$ (resp. $[U_0]\in (D_Y)_0^+$) we have that $V_0$ (resp. $U_0$) is irreducible and generically reduced. 
\end{lemma}

 \begin{proof}
 Its enough to prove the claim for $V_0$, as the one for $U_0$ is the same after switching factors.  With $\sV/C$ as above, the argument of \cite[Lemma~19.1.3]{Fulton98} yields an integral section $v$ of $(R^{2n}\pi_*\Q_{\sX'})^\vee$ with $v_t=[V_t]$.  Using the decomposition of Lemma \ref{lemma limit kunneth}, we in fact have an integral generator $\eta$ of $(j_*R^{2n}\pi_{1*}\Q_{\sZ^{\prime*}})^\vee$ and $v=\eta\otimes \mathrm{pt}$, where $\mathrm{pt}$ is the integral generator of $(j_*\pi_{2*}\Q_{\sY^{\prime*}})^\vee$.  Set
\[
\Gamma:=\bigoplus_{\substack{r+s=2n\\r>0}}(j_*R^r\pi_{1*}\Q_{\sY^{\prime*}})^\vee\otimes (j_*R^s\pi_{2*}\Q_{\sZ^{\prime*}})^\vee.\]

Now, note that $V_0$ is pure-dimensional, and assume that we have an effective decomposition $[V_0]=[V']+[V'']$. In $H_{2n}(X,\Q)$ one has an integral decomposition $[V']=\lambda' v_0+\gamma_0$ (resp. $[V'']=\lambda'' v_0-\gamma_0$) for $\lambda',\lambda'' \in \mathbb Z$ with $1=\lambda'+\lambda''$ and $\gamma_0\in \Gamma_0$, where $\Gamma_0$ is the fiber of $\Gamma$ over the point $0$. Note that the coefficients $\lambda, \lambda'$ are integral since $v_0$ is a primitive vector in the integral cohomology of $X_0$. Generic reducedness is equivalent to the multiplicity of the associated cycle being equal to one, so proving the lemma amounts to showing that either $[V']=0$ or $[V'']=0$.

By Propositions~\ref{proposition decompKE} and~\ref{boundedpotentials}, the Kähler--Einstein metric $\omega_t\in [\theta_t]$ can be decomposed as $\omega_t=\omega_{A_t}+\omega_{B_t}$, and we set $\alpha_t:=[\omega_{A_t}]\in H^2(X_t,\R)$ (resp. $\beta_t=[\omega_{B_t}]\in H^2(X_t,\R)$).  Recall that these currents have bounded local potentials, hence one can take their wedge products, cf. the beginning of this section.  We claim that the cap product 
\begin{equation}
\label{beta}
\beta_0^n \cdot u_0=0 \quad \mbox{for any } u_0\in \Gamma_0.
\end{equation}
Indeed, any such $u_0$ can be written as limit of elements $u_t\in \Gamma_t$ but for $t\neq 0$ we have $\beta_t^n\cdot u_t=0$ almost by definition and $t\mapsto \beta_t^n\cdot u_t$ is continuous by Lemma~\ref{lemma alpha}.  Similarly, we have 
\begin{equation}\label{eq alpha}
\alpha_0^r\beta_0^{n-r}\cdot v_0 =0 \quad \mbox{for any } r>0.
\end{equation}
By Proposition~\ref{boundedpotentials} we have
\[0\leq \int_{V'}(\omega_{B_0})^n=\lambda' \beta_0^n\cdot v_0\]  
and so $0\leq \lambda'$.  By the same argument, we also have $0\leq \lambda''$, and therefore without loss of generality we have $\lambda''=0$; i.e. $[V'']=-\gamma_0$. Combining \eqref{beta} and \eqref{eq alpha}, we find
\begin{align*}
\int_{V''}\om_0^n &= -\underbrace{\beta_0^n\cdot \gamma_0}_{=0} - \sum_{r>0} \alpha_0^r\beta_0^{n-r}\cdot v_0 \\
&=\underbrace{-\sum_{r>0} \int_{V'} (\omega_{A_0})^r \wedge (\omega_{B_0})^{n-r}}_{\le 0}+\underbrace{\sum_{r>0}\alpha_0^r\beta_0^{n-r}\cdot v_0}_{=0}.
\end{align*}
hence $\int_{V''}\omega_0^n=0$, so $[V'']=0$. 
\end{proof}

\begin{proposition}\label{prop restriction}
Let $X$ be as in Setting~\ref{setting relativeKE}, and let $[V_0]\in (D_Z)_0^+$ and $[U_0]\in (D_Y)_0^+$. Then we have
\[ \omega_{A_0}|_{V_0} \equiv 0\quad \mbox{and } \quad  \omega_{B_0}|_{U_0}\equiv 0.\] 
In particular, $\om_0|_{U_0\cap V_0}\equiv 0$. 
\end{proposition}

\begin{proof}
Since the $\omega_{A_0}$ and $\omega_{B_0}$ have bounded local potentials by Proposition~\ref{boundedpotentials}, it makes sense to restrict them to $U_0$ or $V_0$. The last statement follows from the main one thanks to Proposition~\ref{boundedpotentials}.

First assume $V_0\cap X^\reg\neq \varnothing$.  Set $V_0^\circ:=V_0 \cap X^\reg$ which is a dense Zariski open subset of $V_0$. Since $\omega_{A_0}|_{V_0^\circ} \equiv 0$ by Proposition~\ref{proposition decompKE}, it follows that the restriction $\omega_{A_0}|_{V_0}$ is a positive $(1,1)$-current with bounded potentials on $V_0$ supported on $V_0\setminus V_0^\circ$. By Lemma \ref{lemma irred} and the support theorem applied on a resolution say (cf. Remark~\ref{pull back}), the current $\omega_{A_0}|_{V_0}$ has to be a nonnegative combination of currents of integration along the codimension one components of $V_0\setminus V_0^\circ$ which is only possible if $\omega_{A_0}|_{V_0}$ vanishes identically since it has bounded potentials. 

Now for $V_0$ arbitrary, take the germ of a smooth pointed curve $(C,0)$ with a nonconstant morphism $C\to (D_Z)_0$ sending $0$ to $[V_0]$ and consider the pullback diagram
$$\begin{tikzcd}
\sF \ \arrow[r, "e"] \arrow[d,"p"']  & X \\
C &
\end{tikzcd}$$
where we abusively use the same letters $e,p$ for the restrictions of the corresponding maps in \eqref{eq ev}.  By the previous paragraph, $e^*\omega_{A_0}$ vanishes in restriction to $F_t$ for $t$ general. Since $p$ is smooth at a generic point $x_0$ of $F_0=V_0$ by Lemma \ref{lemma irred} and $e^*\omega_{A_0}$ has local bounded potentials, Lemma~\ref{lem psh} below implies that $e^*\omega_{A_0}$ vanishes in restriction to a neighborhood of $x_0$ in $V_0$, hence $e^*\omega_{A_0}|_{G_0}$ is supported on a proper analytic subset. Since its local potentials are bounded, this implies that $e^*\omega_{A_0}|_{G_0}\equiv 0$ as before.
\end{proof}

In the course of the proof above, we used the following lemma about psh functions on the total space of a trivial fibration that are pluriharmonic on all but one slice.

 \begin{lemma}
 \label{lem psh}
 Let $\varphi$ be a psh function on $\Delta^n\times \Delta $. Assume that $\varphi$ is pluriharmonic in restriction to each slice $D_t:=\Delta^n\times \{t\}$ for $t\in \Delta^*$. Then, either $\varphi|_{D_0} \equiv -\infty$ on $D_0$ or $\varphi|_{D_0}$ is pluriharmonic. 
 \end{lemma}
 
 \begin{proof}
 Assume that $\varphi|_{D_0} \not \equiv -\infty$. We use the coordinates $(z_1, \ldots, z_n)$ for the first factor $\Delta^n$ and $(t)$ for the second factor $\Delta$. The lemma is local, and it is sufficient to prove that for every one-dimensional disk $\mathbb D_r^z \subset \Delta^n$ centered at $\underline 0\in \Delta^n$ of small enough radius  $r$, we have 
\begin{equation}
\label{phi 0}
\varphi(0,0)=\fint_{\mathbb D^z_r}\varphi(z,0) dV(z). 
\end{equation}
 For that purpose, we introduce the function $u: \Delta \to \mathbb R \cup \{-\infty\}$ defined by 
 $$u(t):=\fint_{\mathbb D_r^z} \varphi(z,t)dV(z).$$
The function $\underline \varphi(t):=\varphi(0,t)$ is a psh function on $\Delta$ that satisfies 
 \begin{equation}
 \label{u phi}
 u=\underline \varphi \quad \mbox{on } \Delta^*
 \end{equation}
  since $\varphi|_{D_t}$ is pluriharmonic for any $t\neq 0$. To show \eqref{phi 0}, we thus have to extend the identity \eqref{u phi} across the origin. This will be achieved if we show that $u$ is psh on $\Delta$ since 
  psh functions are determined by their value almost everywhere and $\underline \varphi$ is psh. 
  
  First of all, $u$ is upper semicontinuous by Fatou's lemma since $\varphi$ is upper semicontinuous as well (one can assume that $\varphi \le 0$ without loss of generality). Next, if $t_0\in \Delta$ and $s$ is small enough, setting $\mathbb D^t_s:=\{t\in \Delta \mid |t-t_0|<s\}$ we have
  \begin{align*}
  u(t_0)&=\fint_{\mathbb D_r^z} \varphi(z,t_0)dV(z) \\
  &\le \fint_{\mathbb D_r^z}\fint_{\mathbb D_s^t}  \varphi(z,t)dV(t)dV(z)\\
  &=\fint_{\mathbb D_s^t} u(t)dV(t)
  \end{align*}
  where the inequality follows from $\varphi(z,\cdot)$ being psh while the last identity is a simple application of Fubini's theorem. This shows that $u$ is psh, and concludes the proof of the lemma.
 \end{proof}

We will specifically need the following corollary of the previous proposition.  Let $I:=(\sF_Y)_0^+\times_{X}(\sF_Z)_0^+$ be the universal intersection, and consider the resulting diagram
\begin{equation}\label{eq int}\begin{tikzcd}
I^+\ar[d,"g"']\ar[r,"f"]&X\\
(D_Y)^+_0\times (D_Z)^+_0
\end{tikzcd}\end{equation}
where $I^+$ is the unique irreducible component of $I$ dominating $X$, endowed with the reduced structure.

\begin{corollary}\label{corollary one-to-one}In \eqref{eq int}, $f$ is surjective and generically one-to-one, and $g$ is finite, surjective, and generically one-to-one.
 \end{corollary}
 \begin{proof}  The statements for $f$ follow from the fact that the natural maps $(\sF_Y)_0^+\to X$ and $(\sF_Z)_0^+\to X$ are generically one-to-one and surjective by definition and the surjectivity of $g$ is clear as $(\sF_Y)_0^+\to (D_Y)_0^+$ and $(\sF_Z)_0^+\to (D_Z)_0^+$ are surjective. The fibers of $g$ are the intersections $U_0\times_X V_0$ in the notation of Proposition~\ref{prop restriction}, so the finiteness is immediate from there and the K\"ahlerness of $[\omega_0]$.
 
 For the rest of the claim, as $f^{-1}(X^\sing)$ does not dominate $(D_Y)_0^+\times(D_Z)_0^+$, for generic $[U_0]\in (D_Y)_0^+$ and $[V_0]\in (D_Y)_0^+$ we have $U_0\cap V_0\subset X^\reg$.  We may take the germ of a smooth pointed curve $(C,0)$ with a nonconstant map $C\to D_Z$ sending $0$ to $[V_0]$ and dominating $\Delta$, as well as a nonconstant map $C\to D_Y$ sending $0$ to $[U_0]$ and dominating $\Delta$.  If we let $\sF_2$ (resp. $\sF_1$) be the pullback of $\sF_Z$ (resp. $\sF_Y$) to $C$, then both $\sF_1$ and $\sF_2$ are naturally embedded in $\sX':=\sX\times_\Delta C$.  Every component\footnote{When we take intersections we mean the intersections of the reductions.} of $\sF_1\cap\sF_2\cap\sX^{\prime\reg}$ has dimension at least $\dim \sF_1-\codim_{\sX'}\sF_2=(m+1)-m=1$, as for example the intersection is equal to the intersection of $\sF_1\times\sF_2\subset \sX'\times \sX'$ with the diagonal embedding of $\sX'^\reg$, which is regularly embedded in $\sX'^\reg\times \sX'^\reg$.  The fact that $(\sF_1)_t\cap (\sF_2)_t$ is a single point for $t\neq 0$ implies that there is at most one zero-dimensional component of $(\sF_1)_0\cap(\sF_2)_0\cap X^\reg=U_0\cap V_0\cap X^\reg$, which is zero-dimensional.  The claim then follows.
 \end{proof}
 
The following two lemmas will be key to obtain the main result of this section, Proposition~\ref{proposition is product}. This first result below states that, under some reasonable assumptions, a compact Kähler space which is bimeromorphically covered by a nontrivial product is already a product itself. Similarly flavored results have been obtained in the algebraic setting, see e.g. \cite[Proposition~18]{kollar_larsen} and \cite[Lemma~4.6]{Dru18}.

 \begin{lemma}\label{lemma birational product}
Let $X,Y,Z$ be normal compact complex varieties with a bimeromorphic morphism $f:Y\times Z\to X$ and assume $X$ is K\"ahler with rational singularities and $H^1(X,\sO_X)=0$.  Then $X\cong Y'\times Z'$ where $Y'$ (resp. $Z'$) is the normalization of the image of a general fiber $Y\times\{z\}$ (resp. $\{y\}\times Z$).
\end{lemma}
\begin{proof}  By taking resolutions we may assume $Y$ and $Z$ are smooth.  Observe that from the rationality of the singularities and the exact sequence
\[0\to H^1(X,\sO_X)\to H^1(Y\times Z,\sO_{Y\times Z})\to H^0(X,R^1f_*\sO_{Y\times Z})\]
coming from the Leray spectral sequence, we obtain 
\begin{equation}\label{eq product irregularity}
H^1(Y\times Z,\sO_{Y\times Z})=H^1(Y,\sO_Y)\oplus H^1(Z,\sO_Z)=0.
\end{equation}

For a general fiber $Y_z:=Y\times \{z\}$ let $Y''=f(Y_z)$ and $n_Y:Y'\to Y''$ the normalization; likewise for $n_Z:Z'\to Z''$.  Note that the map $g:Y\times Z\to Y'\times Z'$ is birational.  The situation is summarized in the diagram below

$$\begin{tikzcd}
Y \times Z  \arrow[r, "f"]\arrow{d}[swap]{g} & X \\
Y'\times Z' \arrow[dashrightarrow]{ru}&  
\end{tikzcd}$$

For the claim it is enough to show that $f$ factors through $g$ and vice versa; note that such a factorization is obviously unique provided it exists.  As both $f$ and $g$ are proper with connected fibers and map to normal targets, by the rigidity lemma it is sufficient to check the following claim:
\begin{claim} $f$ and $g$ have the same fibers.
\end{claim}
\begin{proof}Let $\omega\in H^2(X,\R)$ be a K\"ahler class $X$.  By \eqref{eq product irregularity} we have that $h^1(Y,\sO_Y)=0$, so the K\"unneth decomposition yields
\begin{equation}
\label{equality kahler}
f^*\omega=\mathrm{pr}_Y^*\alpha+\mathrm{pr}_Z^*\beta
\end{equation}
where $\alpha \in H^2(Y,\mathbb R)$ and $\beta\in H^2(Z,\mathbb R)$. Identifying $Y_z$ with $Y$ (likewise for $Z$), we find that $\alpha=(f^*\om)|_{Y_z}=f^*\omega|_{Y''}$ and $\beta=f^*(\omega|_{Z''})$. Since we have a factorization

$$\begin{tikzcd}
Y_z  \arrow[r, "f|_{Y_z}"]\arrow{d}[swap]{g|_{Y_z}} & Y'' \\
Y' \arrow{ru}[swap]{n_Y}&
\end{tikzcd}$$
we find that $f^*\omega|_{Y''}=g^*n_Y^*\omega|_{Y''}$ and $f^*\omega|_{Z''}=g^*n_Z^*\omega|_{Z''}$. In the end, \eqref{equality kahler} becomes 
\begin{equation}
\label{equality kahler 2}
f^*\omega=g^*(\om_1+\om_2)
\end{equation}
 where $\omega_1=n_Y^*\omega|_{Y''}$ and $\omega_2=n_Z^*\omega|_{Z''}$.  Note that $\omega_1$ and $\omega_2$ are K\"ahler classes by \cite[Theorem~1]{Vaj96}, see also \cite[Proposition~3.6]{GK20}.

Since $\omega$ (resp. $\omega_1+\omega_2$) is a Kähler class, a subvariety $V\subset Y\times Z$ is contracted to a point by $f$ (resp. $g$)  if and only if $f^*\omega|_V=0$ (resp. $g^*(\omega_1+\omega_2)|_V=0$).  The claim now follows from \eqref{equality kahler 2}.
\end{proof}

As explained above, the lemma follows from the claim.\end{proof}

The next lemma allows one to spread a product structure of a given variety to its locally trivial deformations.

\begin{lemma}\label{lemma product persists}
Let $Y,Z$ be compact irreducible and reduced varieties with $H^1(Y,\sO_Y)=0=H^1(Z,\sO_{Z})$.  Let $X=Y\times Z$.  Then the natural map $\Def^\lt(Y)\times\Def^\lt(Z)\to\Def^\lt(X)$ is an isomorphism.  In particular, the universal locally trivial deformation of $X$ is the product of the pullbacks of the universal families of $Y$ and $Z$. 
\end{lemma}

Note that $H^1(X,\R)=0$ in our case so that the hypotheses on $Y$ and $Z$ are fulfilled. 

\begin{proof}  
It suffices to show the map on tangent spaces is an isomorphism and that the map on obstruction spaces is injective.  By the K\"unneth decomposition, the map on tangent spaces is the natural map
\[
H^1(Y,T_Y)\oplus H^1(Z,T_Z)\to H^1(Y,T_Y)\otimes H^0(Z,\sO_Z) \oplus H^0(Y,\sO_Y)\otimes H^1(Z,T_Z)
\]
since the other K\"unneth factors vanish:
\[H^1(Y,\sO_Y)\otimes H^0(Z,T_Z)=0=H^0(Y,T_Y)\otimes H^1(Z,\sO_Z).\]
 In particular, the map on tangent spaces is an isomorphism.  The map on obstruction spaces is likewise identified via the K\"unneth decomposition with the inclusion of the K\"unneth factors
\[H^2(Y,T_Y)\otimes H^0(Z,\sO_Z) \oplus H^0(Y,\sO_Y)\otimes H^2(Z,T_Z)\]
and is therefore injective.
\end{proof}

We are now ready to prove the main result of this section:

\begin{proposition}\label{proposition is product}
Assume Setting~\ref{setting relativeKE}.  Then after shrinking $\Delta$ there is a splitting $\sX=\sY\times_\Delta\sZ$ for locally trivial $\sY,\sZ/\Delta$ restricting to the given splitting over $\Delta^*$.
\end{proposition}
\begin{proof}
The map $g$ of Corollary~\ref{corollary one-to-one} is an isomorphism on normalizations, and so $X$ is a product by Lemma~\ref{lemma birational product}. Now, thanks to Lemma~\ref{lemma product persists}, after shrinking $\Delta$ we have a product decomposition $\sX \isom \sY'\times_\Delta \sZ'$ where $\sY', \sZ'/\Delta$ are locally trivial. Let us denote $Y:=\sY'_0$ and $Z:=\sZ'_0$. By construction, the subspaces $Y \times\{z\}$ and $\{y\}\times Z$ for general $y\in Y$, $z\in Z$ are obtained as fibers of the families $\sF_Y \to D_Y$ respectively $\sF_Z \to D_Z$ which over $\Delta^*$ coincide with $\sX^*\to \sZ^*$ and $\sX^*\to \sY^*$. As one fiber of $\sF_Y \to D_Y$ is contracted under the projection $\sF_Y \to[e]\sX \to \sZ'$, so are all the fibers thanks to the rigidity lemma. In other words, $\sY^* \isom {\sY'}^*$ and likewise for $\sZ^*$.
\end{proof}

\section{Proof of the decomposition theorem}\label{section decomp proof}
In this section we prove Theorem \ref{theorem decomposition}.  Let $X$ be a numerically $K$-trivial K\"ahler variety with log terminal singularities.  By \cite[Corollary~4.2]{CGGN}, there exists a quasi-étale cover $\qeX\to X$ that splits as $\qeX=T\times Y$ where $T$ is a torus and $Y$ is $K$-trivial with vanishing augmented irregularity, $\widetilde q(Y)=0$.  Replacing $X$ with $Y$, we may thus assume $X$ is a $K$-trivial K\"ahler variety with canonical singularities and vanishing augmented irregularity.  Moreover, by Corollary \ref{cor mqe} and \cite[Theorem~C \& Proposition 6.9]{CGGN}, we also know $T_X$ has a splitting  
\begin{equation}
T_{X}=\bigoplus _{i\in I}C_{i} \oplus \bigoplus_{j\in J}S_j\label{decomp sp}
\end{equation}
into foliations such that with respect to some (hence any) singular Ricci-flat metric on $X$ the sheaves $C_i|_{X^{\rm reg}}$ (resp. $S_j|_{X^{\rm reg}}$) are parallel with holonomy $\mathrm{SU}$ (resp. $\mathrm{Sp}$). We may assume \eqref{decomp sp} has at least two factors (or $X$ is IHS or ICY and there is nothing to prove) and at least one symplectic factor (or $X$ is already projective and we apply \cite[Theorem~1.5]{HP19}).  It suffices to show by induction that there is a quasi-\'etale cover of $X$ which splits as a product.

Given Proposition~\ref{proposition is product}, it is enough to show we are in Setting~\ref{setting relativeKE}. We obtain the weakly Kähler metric $\theta$ by \cite[Theorem~6.3]{Bin83i}, see also \cite[Proposition~5]{Nam01b}. It thus remains to show the following:
\begin{claim}
After replacing $X$ with a quasi-\'etale cover, there is a locally trivial deformation $\sX\to\Delta$ of $X$ such that we have a nontrivial product decomposition $\sX^*=\sY^*\times_{\Delta^*}\sZ^*$ over $\Delta^*$ for locally trivial families $\sY^*,\sZ^*/\Delta^*$ and a splitting $T_{\sX/\Delta}=A\oplus B$ compatible with the decomposition over $\Delta^*$.
\end{claim}

\begin{proof}
Taking $E=\bigoplus_{j\in J}S_j$ and $P=\bigoplus_{i\in i}C_i$, then the family $\sX\to S$ guaranteed by Corollary \ref{corollary DefE germ} is a strong locally trivial approximation of $X$ by Lemma \ref{lemma approx crit}.  By Corollary \ref{cor irred constant} every fiber has vanishing augmented irregularity.  A projective fiber has a BB decomposition on some quasi-\'etale cover \cite[Theorem 1.5]{HP19}, so by Lemmas~\ref{lem:cover} and~\ref{lem:cover2} we may assume a fiber $X_t$ of $\sX\to S$ has a BB decomposition which is nontrivial by Corollary \ref{cor irred constant} and which has no torus factor.  By Lemma \ref{lemma product persists} this product decomposition persists in a neighborhood of $t$, and by Lemma \ref{lemma decomposition zariski open} after taking a finite base-change $S'\to S$ and choosing a curve $
\Delta\to S'$ through the special point we have a locally trivial family $\sX\to\Delta$ with a locally trivial product decomposition over $\Delta^*$.  By Proposition \ref{proposition relative splitting} the claim follows.
\end{proof}
The proof of Theorem~\ref{theorem decomposition} is now complete.
\qed



\bibliography{literatur}
\bibliographystyle{alpha}

\end{document}